\documentclass[12pt]{amsart}
\usepackage{amsmath, amsthm, amscd, amsfonts, amssymb, color}
\usepackage{graphicx, subcaption, float} 
\usepackage{dsfont} 
\usepackage{stmaryrd} 
\usepackage[bookmarksnumbered, colorlinks, plainpages]{hyperref}
\usepackage{comment} 
\newtheorem{theorem}{Theorem}[section]
\newtheorem*{theorem*}{T}
\newtheorem{lemma}[theorem]{Lemma}
\newtheorem{prop}[theorem]{Proposition}
\newtheorem{coro}[theorem]{Corollary}

\theoremstyle{definition}
\newtheorem{definition}[theorem]{Definition}

\theoremstyle{remark}
\newtheorem{remark}[theorem]{Remark}
\newtheorem*{remark*}{Remark}
\numberwithin{equation}{section}

\newcommand{\R}{\mathbb{R}}
\newcommand{\N}{\mathbb{N}}

\newcommand{\C}{\mathbb{C}}
\newcommand{\Z}{\mathbb{Z}}
\newcommand{\U}{\mathbb{U}}
\newcommand{\F}{\mathbb{F}}
\newcommand{\E}{\mathbb{E}}
\newcommand{\X}{\mathbb{X}}

\renewcommand{\Re}{\mathfrak{Re}}

\def\Legendre(#1,#2){%
\left(\dfrac{#1 }{#2}\right)
}

\begin{document}
\setcounter{page}{1}

\centerline{}

\centerline{}


\title[The distribution of large values of mixed character sums]{The distribution of large values of mixed character sums}

\author{Amine Iggidr}

\begin{abstract}
In this paper, we investigate the distribution of values of the complete exponential sum 
$S_{p,\chi}(\theta)=\sum_{n=1}^p \chi(n)e(n\theta)$, where $p$ is a large prime, 
$\chi$ is a Dirichlet character (mod $p$) of order $d\geq 2$, and $\theta$ varies over certain 
subsets of $[0,1]$. When $d=2$, these sums correspond to the values of the Fekete 
polynomial associated with $p$ on the unit circle. Our first result gives precise 
estimates for the tail of the distribution of $|S_{p,\chi}(\theta)|$ in a large uniform 
range, when $\theta$ varies over the set $\{(k+1/2)/p\}_{1\leq k\leq p}$. This improves upon a result of Conrey, Granville, Poonen, and Soundararajan. We 
also consider the distribution of the maximum of $|S_{p,\chi}(\theta)|$ for 
$\theta\in I_k=[k/p,(k+1)/p]$, and obtain upper and lower bounds for the distribution 
of large values of this maximum, valid in a uniform range that is nearly optimal: we make this precise in the paper. Our results provide strong support for a conjecture of Montgomery on the maximum 
of Fekete polynomials on the unit circle. In particular, we show that the distribution function 
exhibits double-exponential decay, with a surprising difference in behavior between the 
cases of even and odd order $d$.
\end{abstract}

\maketitle

\section{Introduction}

Let $p$ be an odd prime and $\Legendre(.,p)$ denote the Legendre symbol$\pmod p$.
The \emph{Fekete polynomial} associated to $p$ is defined by

\[
f_p(z) \;=\; \sum_{n=0}^{p-1} \Legendre(n,p)\,z^n.
\]
 Such polynomials appear in Dirichlet’s nineteenth–century work on \(L\)-functions, where he used the identity
\[
L\left(s,\Legendre(\cdot,p)\right) = \frac{1}{\Gamma(s)}\int_0^1 \frac{(-\log t)^{s-1}}{t}\frac{f_p(t)}{1-t^p} \,dt\]
to study \(L(1,\chi)\) and have an analytic expression of \(L(s,(\frac{\cdot}{p}))\). They were named after Michael Fekete, who observed thanks to the above formula that the
absence of zeroes $t\in(0,1)$ of the Fekete polynomial implies an absence of real zeroes $s>0$ for \(L(s,(\frac{\cdot}{p}))\). 
Such polynomials have been extensively studied; see, for instance, \cite{BorweinChoiFekete, ConreyGranvillePoonenSoundararajan2000, ErdelyiFekete, klurman2023lqnormsmahlermeasure}. 

\medskip

Montgomery~\cite{MontgomeryLegendre} initiated the study of extremal values of Fekete polynomials on the unit circle, establishing upper and lower bounds for
\(
\displaystyle \max_{\theta\in[0,1]} \bigl|f_p\big(e( \theta )\big)\bigr| \), where here and throughout we let \( e(t):=e^{2\pi i t}.
\)
More precisely, he proved
\[
\sqrt{p}\log\log p \ll \max_{\theta\in[0,1]} \bigl|f_p\big(e( \theta )\big)\bigr| \ll \sqrt{p}\log p, \qquad \text{as } p\to\infty.
\]
Montgomery further 
conjectured that the true order of growth of this maximum is given by the lower bound, namely
\begin{equation} \label{eq:conjMontgomery}
    \max_{|z|=1} |f_p(z)| \asymp \sqrt{p}\log\log p.
\end{equation}

The work of Conrey, Granville, Poonen and
Soundararajan~\cite{ConreyGranvillePoonenSoundararajan2000} placed Montgomery's
approach in a broader context, combining probabilistic tools with analytic
techniques to study the zero distribution and large‐value asymptotics of
\(f_p\). In particular, one of their results can be viewed as progress toward
Montgomery’s conjecture, by describing the limiting distribution of large
values of \(f_p\) at midpoints of the arcs
between consecutive \(p\)-th roots of unity. 
For a prime $p$, we denote by $e_p(x) := e^{2i \pi x/p}$ the canonical additive character of 
$\F_p$ (where $\F_p$ is the finite field of order $p$). Their theorem~\cite[Theorem 4]{ConreyGranvillePoonenSoundararajan2000} may be stated as follows. 
Let \(p\) be a prime, for every fixed real number \(V\), we have

\[
\frac{1}{p}\left|\Bigl\{ k \in \{1,2,\dots,p\} :  \left| f_p\Bigl(e_p(k+\tfrac{1}{2})\Bigr) \right|
> V\sqrt{p} \Bigr\}\right|
\sim c_V, \qquad \text{as } p \to \infty,
\]
where
\( c_V =  \exp(-\exp(\pi V /2 + O(1)))\) for large fixed $V$. The resulting distributional statement for the values of \(f_p\) therefore provides quantitative 
information on how often \(f_p\) takes large values, in line with the philosophy underlying
Montgomery’s conjecture. 

\medskip 
In the present paper, we strengthen this result in two complementary
directions. First, we derive a substantially more precise asymptotic formula, in which the
\(O(1)\) term is replaced by an explicit constant together with a
significantly sharper error term. Moreover, we obtain this estimate in a wide range of $V$,
which we believe to be nearly optimal, in accordance with Montgomery's conjecture. This leads to the following result.

\begin{theorem}
    \label{theorem:distribution1/2CharacterSum}
   Let $p$ be a large prime.
    For all real numbers $ 1\leq V \leq \frac{2}{\pi}(\log\log p - 2\log\log\log p)$ we have

    \begin{equation*}
        \dfrac{1}{p} \left|\left\{K \in \F_p : \; \dfrac{1}{\sqrt{p}}\left| f_p\Bigl(e_p(K+\tfrac{1}{2})\Bigr) \right| \geq V \right\} \right|
        = \exp\left( - C_2^- \exp \left(\frac{\pi}{2}V  \right)  \left(  1 + O(e^{-\pi V/4 }) \right)   \right) 
    \end{equation*}
\noindent

\[
\text{where} \; C_2^-= \frac{2}{\pi} \exp\left( - \gamma - 1 + \log \frac{\pi}{2} - 
\displaystyle\int_0^1 \frac{\log \cosh (u)}{u^2}\,du - \displaystyle\int_1^{\infty} \frac{\log\!\left(1+e^{-2u}\right)}{u^2}\,du \right),
\]
and $\gamma$ is the Euler–Mascheroni constant.
\end{theorem}
Second, we go beyond the midpoint analysis by considering the maximum of
\(|f_p|\) over the entire subarc between two consecutive \(p\)-th roots of
unity, rather than only at the midpoint. More precisely, we investigate the distribution of 
\( \displaystyle\max_{t\in [0,1]} \left| f_p\left(e(\tfrac{K+t}{p})\right) \right| \)
as $K$ varies in $\F_p$. One can observe that 

\begin{equation}\label{eq:max=doublemax}
\displaystyle\max_{\theta\in[0,1]} \left|f_p\big(e( \theta )\big)\right| = 
\displaystyle\max_{K\in \F_p}\displaystyle\max_{t\in [0,1]} \left| f_p\left(e(\tfrac{K+t}{p})\right) \right|,
\end{equation}
\noindent
so that investigating the extreme values of this quantity will shade light on Montgomery's conjecture.

\medskip
More recently, Wang and Xu~\cite{WangXuMixedCharSum} investigated the average size of mixed character sums, specifically of the form 
$\sum_{1 \leq n \leq x} \chi(n)e(n\theta)w(\tfrac{n}{x})$ as $\chi$ varies over Dirichlet characters (mod $q$), where $w$ is a smooth function. They found that these sums are on the order of 
$\sqrt{x}$ when $\theta$ is an irrational real number satisfying a weak Diophantine 
condition. This result highlights a distinction from previous findings by 
Harper~\cite{harper2023typicalsizecharacterzeta}, where for rational $\theta$, the average size was found to be $o(\sqrt{x})$.
Another recent development in the study of mixed character sums is due to Bober, Klurman, and Shala~\cite{BoberKlurmanShala2025} who revisited this circle of ideas from a broader perspective, by investigating the distributional behaviour of 
\emph{mixed character sums}
\[
S(\chi, x, \theta) \;=\; \sum_{n\le x} \chi(n) e(n\theta),
\]
in various averaging regimes. Their distributional results allows them to construct an infinite family of Littlewood polynomials with record-large Mahler measure
(advancing the Mahler problem).

\medskip

Motivated by these developments, we define for any nonprincipal Dirichlet character \(\chi \pmod p\) the associated generating polynomial
\begin{equation}
    \label{def:fchi}
f_\chi(z)\;=\;\sum_{n=0}^{p-1}\chi(n)\,z^n.
\end{equation}

Our definition generalizes the Fekete polynomial, which corresponds to the special case where $\chi$ is a quadratic character.  
The study of $f_\chi$ for characters of arbitrary order thus extends Montgomery’s framework to a wider family of multiplicative characters, allowing one to probe 
the same oscillatory behavior in greater generality.

\medskip

In this paper, we investigate the distribution of $f_\chi$ on the unit circle, namely the exponential sum
\[
f_\chi(e(\theta)) = \sum_{n\in\F_p}\chi(n)e(n\theta), \qquad \theta\in[0,1].
\]  
Throughout our work, we will normalize the above sum by the factor $p^{-1/2}$. Indeed, by comparison with the size of Gauss 
sums, we expect $|f_\chi(e(\theta))|$ to be typically of order $\sqrt{p}$, so dividing by this factor places $f_\chi$ on its natural probabilistic 
scale, in accordance with the heuristic analogy between $f_\chi(e(\theta))$ and a random walk of $p$ unit steps in the complex plane.

\medskip

To study the frequency of these large deviations, we partition the unit circle into $p$ equal subarcs between consecutive $p$-th roots of unity, 
namely between $e_p(K)$ and $e_p(K+1)$ for $K\in\F_p$. This
decomposition is natural in view of the identity~\eqref{eq:max=doublemax}. For $x=0$ and $x=1$, the expression reduces to a Gauss sum, thus these values will not contribute in the large deviations.
We therefore consider the distribution of the maximum of $\left| f_\chi \right|$ on the $K^\text{th}$ arc as $K$ varies, that is,

\[
\Phi_\chi(V)\ :=\ \frac{1}{p}\,\Bigl|\Bigl\{K\in\F_p:\; \max_{x\in[0,1]}\Bigl|\frac{f_\chi(e_p(K+x))}{\sqrt p}\Bigr|\ \ge V\Bigr\}\Bigr|.
\]

This quantity measures the proportion of subarcs on which the normalized generating polynomial exceeds a value $V$, and serves as the central 
object of our study. We will prove matching (up to explicit constants) \emph{double-exponentially decreasing} 
upper and lower bounds for $\Phi_\chi(V)$ in the uniform range
\begin{equation}
\label{def:deltad}
1 \le V\le \dfrac{\delta_d}{\pi}(\log_2 p-2\log_3 p),
\qquad \text{ where } \delta_d \;:=\;
\begin{cases}
2 \cos\!\dfrac{\pi}{2d}, & \text{if $d$ is odd;} \\[1.2ex]
2, & \text{if $d$ is even;}
\end{cases}
\end{equation}
\noindent
where we adopt the following usual notation for iterated logarithms
\[
\log_2 x := \log \log x, \qquad \log_{n+1} x := \log_n(\log x), \qquad  n\ge 2.
\]

\medskip

The work of Granville and Soundararajan on large character sums
\cite{GranvilleSoundararajan2007} reveals a pronounced dichotomy between
characters of \emph{odd order} and \emph{even order} with respect to the
maximum of their sums. More precisely, for primitive characters of bounded
odd order and under the Generalized Riemann Hypothesis, they showed that 
the maximum is a $O\!\left(\sqrt{p}(\log\log p)^{1-\gamma_d} \right)$ with $0<\gamma_d<1$, 
whereas the bound $\sqrt{p}\log \log p$ is attained infinitely often for even-order characters 
according to R.E.A.C Paley's 1932 result~\cite{paley1932}.
Strikingly, a similar parity-dependent phenomenon emerges in our
setting: we observe a marked discrepancy between the
even and odd cases, which is reflected in the geometry of the values of
$\chi$ on the unit circle and, quantitatively, in the scale of our tail
bounds through the parameter $\delta_d$.

\medskip

Our main results are the two following theorems.

\begin{theorem}
    \label{theorem:distributionEvenOrderCharacterSum}
   Let $p$ be a large prime. Uniformly for all Dirichlet characters $\chi\pmod p$ of even order $d$ and
    for all real numbers $ 1\leq V \leq \frac{2}{\pi}(\log_2 p - 2\log_3 p)$ we have
    \begin{equation}\label{eq:thmEvenOrders}
        \Phi_\chi(V) = \exp\left( - \exp \left(\frac{\pi}{2}V +O(1) \right) \right), 
    \end{equation}
\noindent
where the implicit constant is absolute and lies between $2.4 \times 10^{-6}$ and $0.3207$.

\end{theorem}

\begin{remark*}
    What we prove is actually slightly stronger, we get explicit constants for the $O(1)$ in the upper and lower bound, 
    respectively in Theorem~\ref{thm:UpperBoundBothTheorems} and in Theorem~\ref{thm:LowerBoundEvenThm}, with the lower bound expected
    to be the true order of magnitude. Moreover, for the lower bound, the explicit constant is decreasing in $d$ and for large $d$, it is closer to $0.3070$;
    this behavior is illustrated by the numerical experiments in Figure~\ref{fig:twoplots}.
\end{remark*}

\medskip
In particular, if the estimate~\eqref{eq:thmEvenOrders} holds up to $V_{max}(p):=\displaystyle\max_{\theta\in[0,1]}\left|\frac{f_\chi(e(\theta))}{\sqrt p} \right| $, 
then this would imply that $\Phi_\chi(V_{max}(p))\geq \frac{1}{p}$, thus $V_{max} \leq \frac{2}{\pi} \log \log p$ + O(1) 
which would prove a stronger form of Montgomery’s conjecture~\eqref{eq:conjMontgomery}, and this shows that our range is nearly optimal.

\medskip

The counterpart of Theorem~\ref{theorem:distributionEvenOrderCharacterSum} for odd-order characters reads as follows:

\begin{theorem}
    \label{theorem:distributionOddOrderCharacterSum}
    Let $d\geq 3$ be a fixed odd integer and $p \equiv 1\pmod d$ be a large prime. Uniformly for all Dirichlet 
    characters $\chi\pmod p$ of odd order $d$ and
    for all real numbers $ 1\leq V \leq \frac{2}{\pi}\cos\frac{\pi}{2d}(\log_2 p - 2\log_3 p)$, we have

\begin{equation*}
    \Phi_\chi(V) = \exp\left( - \exp \left(\frac{\pi}{2\cos\!\frac{\pi}{2d}}V +O_d(1) \right) \right).
\end{equation*}

\noindent

\end{theorem}

\begin{remark*}
    Here, the implicit constant depends on $d$, we derive an explicit upper and lower bound for $\Phi_\chi(V)$ respectively in
    Theorem~\ref{thm:UpperBoundBothTheorems} and in Section~\ref{section6:proofoddLowerBound}, Equation~\eqref{eq:valCd-smallOdd}. 
    It is worth noting that the lower bound increases too quickly with the order $d$, although the techniques 
    of Section~\ref{section5:proofLowerBound} can be adapted and yield similar asymptotics when $d \gg (\log p)^2$.

\end{remark*}

\begin{coro}
    For all large primes $p$, for all Dirichlet characters $\chi \pmod p$ of order $d$, there exists $\theta\in[0,1]$
    such that if $d$ is even,
    \[ \left| \sum_{n\in\F_p}\chi(n)e(n\theta) \right| \geq \left(\dfrac{2}{\pi} +o(1) \right)\sqrt{p}\log\log p,
    \]
    whereas if $d$ is odd,
    \[ \left| \sum_{n\in\F_p}\chi(n)e(n\theta) \right| \geq \left(\dfrac{2}{\pi}\cos\frac{\pi}{2d} +o(1) \right)\sqrt{p}\log\log p.
    \]

\end{coro}
\medskip
\begin{remark}
It is plausible that the maximum over $\theta$ in the case of even characters exceeds that of odd characters. 
This is in line with the results of Granville and Soundararajan~\cite{GranvilleSoundararajan2007} on large character sums, where the parity of the order of the character plays an important role in determining the maximum of the partial sums of the character sum.
\end{remark}

\begin{remark}
\textit{Turyn polynomials} are defined by applying a cyclic shift to the coefficients of a 
Fekete polynomial. They are Littlewood polynomials except for one vanishing coefficient, 
and have been studied in the context of extremal problems for polynomial norms on the unit 
circle, see, for instance, \cite{BorweinChoiFekete,Gunther_Schmidt_2017,mossinghoff2025mahlersproblemturynpolynomials}.
The same results we got in Theorems~\ref{theorem:distributionEvenOrderCharacterSum} and~\ref{theorem:distributionOddOrderCharacterSum} 
for $f_\chi$ may be obtained for generalized Turyn polynomials defined as follows for a shift 
$a\in\Z$: $f_{\chi,a}(z) = \sum_{n=0}^{p-1} \chi(n+a)\, z^n$. 
Indeed, one may carry out the entire analysis with a suitable
modification of the auxiliary function \eqref{def:g-auxiliary}, replacing \(g_{\chi,K}\) by 
    \begin{equation}
        \label{def:g-auxiliary-shiftedChi}
  g_{\chi,K,a}(x) :=  \dfrac{ie_p(Ka)}{z^p-1}\dfrac{f_\chi(z)}{\tau(\chi)},
    \end{equation}
    where $z:=e_p(K+x), \; K\in\Z, \, x\in(0,1)$. All the methods developed in this paper remain valid in this
more general setting, with only minor notational changes.
\end{remark}

\subsection{Organization of the paper}

The paper is organized as follows. In the next section, we recall the necessary analytic tools, introduce the
auxiliary function \(g_{\chi,K}\) associated with \(f_\chi\), and collect
some elementary lemmas needed for the proof of the lower and upper bounds
in our theorems. Section~\ref{section3:proofUpperBound} is devoted to the
proof of the upper bounds for both Theorems~
\ref{theorem:distributionEvenOrderCharacterSum} and~
\ref{theorem:distributionOddOrderCharacterSum}. In
Section~\ref{section4:Part1LowerBound}, we set up the main tools for proving
the lower bounds, including the link between the arithmetic model and the
random one. In Section~\ref{section5:proofLowerBound}, we use the saddle point
 method to derive the lower bound in the even-order case and for large 
 orders.

We should emphazise that we encounter a technical difficulty in the small odd-order case: Taking only the
real part in Section~\ref{section4:Part1LowerBound} leads to a weaker result due to the even--odd
discrepancy. While this loss becomes apparent in
Section~\ref{section5:proofLowerBound}, it does not affect the final scale of
the bounds for large odd order \(d\). For small values of \(d\), however,
additional ideas are needed, and the corresponding lower bounds are proved
separately in Section~\ref{section6:proofoddLowerBound} using an independence
argument for the values of Dirichlet characters in short intervals.

\subsection{Notations and conventions} 
  We write $A \ll B$ or $A = O(B)$ to denote the 
estimate $|A| \leq CB$ for some constant $C$, and $A = o(B)$ to denote the bound $|A| \leq c(p) B$ for some 
$c(p)\underset{p\rightarrow \infty}{\rightarrow} 0$.  We write $A \asymp B$ for $A \ll B \ll A$. 
Unless specified otherwise, asymptotic quantities such as $o(1)$ or $O(1)$ are interpreted as 
$p\rightarrow \infty$, and summations will run over integers. Moreover, $\mathds{1}(A)$ denotes the indicator function of an event $A$.

\subsection{Acknowledgments}
I sincerely thank my PhD advisors, Professor Youness Lamzouri and Professor Thomas Stoll, for their guidance and 
valuable insights throughout the preparation of this article.

\section{Analytic tools}
\label{section:AnalyticTools}

We begin by writing down some analytic tools and standard notation that will be used throughout the paper.  
In particular, we recall the definition of the Gauss sum associated with a Dirichlet character $\chi \pmod p$, given by
\[
\tau(\chi)\ :=\ \sum_{n=1}^{p} \chi(n)\,e_p(n).
\]
It is well known that $|\tau(\chi)|=\sqrt{p}$ and that $f_\chi(e^{2i\pi k /p})= \overline{\chi}(k)\tau(\chi)$ for $k\in \F_p$ 
when $\chi$ is primitive~\cite[Theorems 9.5 and 9.7]{Montgomery-multNT}. However, there is no general technique for computing the complex 
argument of a Gauss sum, except in certain specific cases, such as when $\chi$ is the Legendre symbol.

\medskip

Following the ideas of Conrey et al.~\cite[Equation (2.3)]{ConreyGranvillePoonenSoundararajan2000} in their construction of an auxiliary 
function for Fekete polynomials, we introduce a generalization of this \emph{auxiliary exponential sum} closely related to the generating 
polynomial $f_\chi$ defined in~\eqref{def:fchi}, and whose moments can be estimated more efficiently. This construction will play a central role 
in the proof of the upper and lower bounds in Sections~\ref{section3:proofUpperBound} and~\ref{section5:proofLowerBound}.

\begin{definition}
    For $K\in \F_p$, $x\in (0,1)$ and $\chi \pmod p$ a non-principal Dirichlet character, we write $z=e_p(K+x)$
     and let

\begin{equation} 
    \label{def:g-auxiliary}
 g_{\chi,K}(x) :=  \dfrac{i}{z^p-1}\dfrac{f_\chi(z)}{\tau(\chi)} .
\end{equation}

\end{definition}

We now express $g_{\chi,K}$ in a form that will be more convenient for our subsequent arguments.
\begin{lemma}
    With the notation of the previous definition, we have

\begin{equation} 
    \label{formulaG} g_{\chi,K}(x)  = \frac{i}{p}\displaystyle\sum_{|j| < p/2} \dfrac{\overline{\chi}(K-j)}{e_p(j+x) -1}.
\end{equation}

\end{lemma}

\begin{proof}
By Lagrange's interpolation, we have the following:
\[ f_\chi(z) = \sum_{k=0}^{p-1} f_\chi(e_p(k)) \prod_{\substack{j=0 \\ j \neq k}}^{p-1} \left( \frac{z - e_p(j) }{e_p(k) - e_p(j) } \right). \]
\noindent
Moreover,
\[ \prod_{\substack{j=0 \\ j \neq k}}^{p-1} (z - e_p(j) ) = \frac{z^p - 1}{z - e_p(k)}, \]
\noindent
and 
\[ \prod_{\substack{j=0 \\ j \neq k}}^{p-1} (e_p(k) - e_p(j) ) = e_p(k(p-1))\prod_{j=1}^{p-1} (1 - e_p(j) ) = p e_p(-k). \]
\noindent
Thus, 

\begin{equation}
    \label{formulefp} 
    f_\chi(z) = \sum_{k=0}^{p-1} \dfrac{ f_\chi(e_p(k))e_p(k)}{p}  \frac{z^p - 1}{z - e_p(k)} .
\end{equation}
Therefore, if $K\in \F_p, x\in (0,1)$ and $z=e_p(K+x)$, by (\ref{formulefp}) and the expression of the Gauss sum, we derive the second expression of $g_{\chi,K}$,
\begin{equation}
    \label{calculProof1}
     \frac{i}{z^p-1} \frac{f_\chi(z)}{\tau(\chi)} =  \frac{i}{p}\sum_{k=0}^{p-1} \overline{\chi}(k)\frac{e_p(k)}{z - e_p(k)}
      =  \frac{i}{p}\sum_{|j|<p/2} \overline{\chi}(K-j)\frac{1}{e_p(j+x) - 1}.
\end{equation}

\end{proof}

The Weil bound for character sums, which is a consequence of Weil’s proof of the Riemann Hypothesis for curves over
 finite fields (see \cite{Montgomery-multNT}, Lemma 9.25 for a proof of the following statement), will be a key tool for computing the 
 moments of $g_{\chi,K}$.

\begin{lemma}{\textbf{(Weil)}}
 \label{WeilBound}
    Let $d\in \N$, $p$ be a prime such that $d \mid(p-1)$ and $\chi$ be a character \textnormal{(mod $p$)} of order $d$. If $f\in \mathbb{F}_p[X]$ cannot 
    be written as $g^d \bmod p$ for some $g\in\F_p[X]$, then for $k$ being the number of distinct roots of $f$, the following holds:
    \[  \left| \displaystyle\sum_{x=1}^p \chi(f(x)) \right| \leq (k-1)\sqrt{p}.  \]
    
\end{lemma}

\medskip

\section{Proof of the upper bound}
\label{section3:proofUpperBound}

In this section, we will prove the following results:

\begin{theorem}
    \label{thm:UpperBoundBothTheorems}
    Let $p$ be a large prime. Uniformly for all Dirichlet characters $\chi\pmod p$ of order $d$ and
    for all real numbers $ 1\leq V \leq \frac{2}{\pi}(\log_2 p - 2\log_3 p)$, we have

\begin{align}
    \Phi_\chi(V) &\leq  \exp\left( -C_d^+ \exp \left(\frac{\pi}{2}V  \right) \right),
\end{align}
where $C_d^+ = 28\cdot 10^5\exp\left(-\dfrac{\pi}{\delta_d} -\gamma  \right)$ with $\delta_d$ defined in~\eqref{def:deltad}.

\end{theorem}

By using \eqref{formulaG}, we can extend by continuity the function $x\mapsto 2\sin(\pi x)g_{\chi,K}(x)$, $x\in (0,1)$ to $[0,1]$, 
and this allows us to rewrite $\Phi_\chi(V)$ as
\[
\Phi_\chi(V)\ =\ \frac{1}{p}\,\Bigl|\Bigl\{K\in\F_p:\ \max_{x\in[0,1]}\left|2\sin(\pi x)g_{\chi,K}(x)\right|\ \ge V\Bigr\}\Bigr|,
\]
so in order to prove the upper bounds of Theorems~\ref{theorem:distributionEvenOrderCharacterSum} and~\ref{theorem:distributionOddOrderCharacterSum},
we first smooth the sum $g_{\chi,K}$ and then split it at a parameter $J$ to be chosen later,
\begin{align} \label{eq:gpkSplitSum}
g_{\chi,K}(x) = \frac{1}{2\pi}\sum_{|j| \leq J} \dfrac{\overline{\chi}(K-j)}{j+x} + \frac{i}{p}\sum_{J < |j| < p/2} \dfrac{\overline{\chi}(K-j)}{e_p(j+x)-1} + O\!\left( \frac{J}{p}\right).
\end{align}

To derive this expression, one could use the following estimate, which holds uniformly for $|j| < p/2$:
\begin{equation}
    \label{eq:approxCoeffSerieFourier}
\frac{i}{p(e_p(j+x))-1} = \frac{1}{2\pi}\frac{1}{j+x} + O\!\left(\frac{1}{p}\right).
\end{equation}

We shall bound the sum over the range \(\{|j|\le J\}\) by replacing the character with random values, and we will control the size of the second
 sum (over the range \(\{J<|j|< p/2\}\)) using the method of moments.

\medskip

First, we will need a few useful lemmas in order to derive the upper bounds of Theorems~\ref{theorem:distributionEvenOrderCharacterSum} and~\ref{theorem:distributionOddOrderCharacterSum}.
\begin{lemma}
    \label{lemma:HarmonicEstimation}
    For $n\in \N$ and $x\in (0,1)$, let $H_n(x):= \displaystyle \sum_{k=0}^{n-1} \dfrac{1}{k+x}$.
    \newline
    The following holds for all $x\in(0,1]$ as $n\rightarrow\infty$,
    \[
    H_n(x) = \log n - \psi(x) + O\!\left(\frac{1}{n}\right),
    \]
    where $\psi$ is the Digamma function~\cite[p.258]{AbramovitzStegun-HandbookOfMathFunction}.
    Moreover, 
    \[ -\psi\left(\frac{1}{2}\right) =  \gamma + 2\log 2,
    \]
    where $\gamma$ is the Euler–Mascheroni constant, and for small deviations 
    around $\frac{1}{2}$, if $\varepsilon_n = O\!\left(\frac{1}{\log n}\right)$, we have the following\textnormal{:}  
    \[ H_n\left(\frac{1}{2} + \varepsilon_n\right) = H_n\left(\frac{1}{2}\right) + O\!\left(\frac{1}{\log n}\right).
    \]
\end{lemma}

\begin{proof}
    From \cite{AbramovitzStegun-HandbookOfMathFunction}, p.258, equation (6.3.6), we get that for $n\geq 1, x\in (0,1]$,
   \[ H_n(x) = \psi(n+x) - \psi(x).
   \]
    The first assertion follows from the asymptotic development of $\psi$ 
    \cite[p.259, equation (6.3.18)]{AbramovitzStegun-HandbookOfMathFunction}, whereas the second assertion comes from equation (6.3.3), p.258.
    
    As for the last one, since $\psi'\left( \frac{1}{2}\right)= \frac{\pi^2}{2}$ \cite[p.260, equation (6.4.4)]{AbramovitzStegun-HandbookOfMathFunction},
    a Taylor expansion of $\psi$ around $\dfrac{1}{2}$ gives us the desired estimate.
\end{proof}

Let $\U_d := \{e_d(k) :\; 0\le k < d\}$ in the statement of the next lemma.

\begin{lemma}\label{lemma:maxOfSumUnitCircle}

Let $n \ge 1$ and $d\ge 2$. Then the following holds as $n\rightarrow\infty$\textnormal{:} 

\[
\max_{\substack{x\in(0,1) \\ (a_j)_{|j|\leq n}\in \U_d^{2n+1}}} \left| \sin(\pi x)\sum_{|j| \leq n} \dfrac{a_j}{j+x } \right| =
\begin{cases}
 2\cos\frac{\pi}{2d}\left( \log n + \gamma + 2\log 2 + o(1)\right), \text{ if $d$ is odd;}
 \\
 2\log n + 2\gamma + 4\log 2 + o(1), \text{ if $d$ is even;}
\end{cases}
\]
 and the implicit constant is absolute, i.e., it is independant of $d$.
\end{lemma}

\begin{proof}
For even $d$, since $-1\in \U_d$, the maximum is reached when $a_j=\mathrm{sgn}(j+x)$ for all $j\in \Z$, and the result follows from Lemma~\ref{lemma:HarmonicEstimation}.

\medskip

For odd $d$, to compute the absolute value in the statement of the lemma we will use the fact that for $z\in\C$, we have

\[ |z| =  \max_{ \theta \in [0,2\pi] } \Re(e^{-i\theta}z). \]
\noindent
Fix \(\theta\in [0,2\pi]\) and write
\[
\delta=\delta(\theta):=\min_{0\leq k < d}\left|\,\theta-2\pi k/d\,\right|\in[0,\pi/d].
\]
Then
\[
\max_{u\in\mathbb U_d}\Re(e^{-i \theta}u)=\cos\delta.
\]
For the minimum over \(\mathbb U_d\), since $d$ is odd, the closest angle to \(\theta+\pi\) is at distance \(\pi/d-\delta\), hence
  \[
  \min_{u\in\mathbb U_d}\Re(e^{-i \theta}u)=-\cos(\pi/d-\delta).
  \]
This gives us
\begin{align*} 
 & \max_{ (a_j)^\Z\in \U_d^\Z} \Re\left(  \sum_{0\leq j \leq n} \dfrac{e^{-i \theta}a_j}{j+x } - \sum_{0 \leq j <n} \dfrac{e^{-i \theta}a_{-1-j}}{j+1-x } \right) \\
 = \;&  \cos( \delta) H(n+1,x) + \cos(\pi/d - \delta)H(n,1-x).
\end{align*}

For large \(n\), the presence of the factor $\sin(\pi x)$ forces the maximizer in $x$ 
to occur at \(x=\frac{1}{2}+O\!\left(\frac{1}{\log n}\right)\).
Thus we want to maximize
\[
\cos\delta + \cos(\pi/d - \delta),
\]
and the maximum is attained at $\delta = \frac{\pi}{2d}$. The result follows by using Lemma~\ref{lemma:HarmonicEstimation}.

\end{proof}

\medskip

Having obtained a bound for the sum over the range $\{|j|\le J\}$ in Lemma~\ref{lemma:maxOfSumUnitCircle} 
(by taking $n=J$), we now turn to the second sum, over the range $\{J<|j|< p/2\}$. 

\medskip
Before that, we shall define, for 
$\textbf{j} =(j_1,\ldots,j_{k})$ a $k$-tuple and $1\leq k_1 \leq k$, the following polynomial:

\begin{equation}
    \label{eq:def_polynomeTuple}
    f_{\textbf{j},k_1}(X) := \prod_{h=1}^{k_1} (X-j_h)^{d-1}\prod_{h=k_1+1}^{k}(X-j_h) .
\end{equation}

\begin{prop}
    \label{prop:MajorationSommeMN}
    Let $\{c_j\}_{j\in \Z}$ be a sequence of complex-valued functions of a variable $x\in [0,1]$ such that $|c_j(x)| \leq \frac{C_0}{|j|-1}$ 
    for $|j| \geq 2$, where $C_0$ is a positive constant. Let $d\geq 2$ be an integer, $p\equiv 1 \pmod d$ a large prime and 
    $1\leq M < N < \frac{p}{2}$ be positive integers. If $\chi \pmod p$ is a character of order $d$, then, for all positive integers $k$, we have
    
\begin{equation*}
    \frac{1}{p} \sum_{K\in\F_p} \left| \sum_{M < |j| < N}  \overline{\chi}(K-j)c_j(x) \right|^{2k}   \leq  (16C_0^2)^{k}  \sum_{l=1}^k \left(\dfrac{l}{M} \right)^{2k-l} + \dfrac{2k}{\sqrt{p}} \left(2C_0\log\frac{N}{M} \right)^{2k}  .
\end{equation*}
\noindent
Moreover, if $k\leq \frac{M}{e}$, we can simplify the right hand side\textnormal{:} 
\begin{equation}
    \frac{1}{p} \sum_{K\in\F_p} \left| \sum_{M < |j| < N}  \overline{\chi}(K-j)c_j(x) \right|^{2k}   \leq  \left(\dfrac{32C_0^2k}{M} \right)^{k}  + \dfrac{2k}{\sqrt{p}} \left(2C_0\log\frac{N}{M} \right)^{2k}  .
\end{equation}

\end{prop}

\begin{proof}

For convenience, we write 
\[
S:= \frac{1}{p}\sum_{K\in\F_p} \left| \sum_{M <|j| < N}  \overline{\chi}(K-j)c_j(x) \right|^{2k} . \] 

\noindent
We have
\[
S=  \frac{1}{p}\sum_{K=1}^p \left( \sum_{M <|j| < N}  \overline{\chi}(K-j)c_j(x) \right)^k  \left(\sum_{M < |j| < N}  
    \chi(K-j)\overline{c_j(x)}\right)^k,
\]

\noindent
so by expanding the power, switching the sums and using $\overline{\chi} = \chi^{d-1}$, we get

\begin{align*}
  S&= \dfrac{1}{p} \sum_{M < |j_1|, \ldots, |j_{2k}| <N}  \prod_{h=1}^k c_{j_h}(x)\overline{c_{j_{h+k}}(x)}  
    \sum_{K=1}^p\chi\left(\prod_{h=1}^k (K-j_h)^{d-1}\prod_{h=k+1}^{2k}(K-j_h)  \right)  \\
    &=S_1 + S_2,
\end{align*}
\noindent
where we have split the sum into two parts: one over the $2k$-tuple $(j_1,\ldots,j_{2k})$ for which 
$f_\textbf{j}(x) := f_{\textbf{j},k}(x) = g^d \bmod p$ for some $g\in\F_p[X]$ (see ~\eqref{eq:def_polynomeTuple} for the definition of $f_{\textbf{j},k}$), 
and the other over the elements for which $f_\textbf{j}$ is not a perfect $d$-th power mod $p$. We will note them 
respectively $S_1$ and $S_2$. We will furthermore define $\mathcal{J}$ as the set of all $2j$-tuples $(j_1,\ldots,j_{2k})$ in 
the summation:
\[\mathcal{J}:= \{(j_1,\ldots,j_{2k}) \mid \forall \; 1 \leq i \leq 2k,\; M < |j_i| < N\}.
\]
Let $\mathcal{J}_1$ be the subset of $\mathcal{J}$ such that the polynomial $f_\textbf{j}$ is a perfect $d$-th power and $\mathcal{J}_2 = \mathcal{J} \setminus \mathcal{J}_1$.

\medskip

The sum $S_2$ can easily be bounded by using Weil's bound for character sums (Lemma~\ref{WeilBound}) as follows:

\begin{align*}
|S_2| &\leq \dfrac{1}{p} \sum_{\textbf{j}\in \mathcal{J}_2} \left|\prod_{h=1}^k c_{j_h}(x)\overline{c_{j_{h+k}}(x)}  \sum_{K=1}^p\chi\left(\prod_{h=1}^k (K-j_h)^{d-1}\prod_{h=k+1}^{2k}(K-j_h)  \right) \right| \\
&\leq \dfrac{1}{p} \sum_{\textbf{j}\in \mathcal{J}} \left( \prod_{h=1}^{2k} \frac{C_0}{|j_h|-1}\right)  2k \sqrt{p}\\
&\leq \dfrac{2k}{\sqrt{p}} \left(2C_0\log\frac{N}{M} \right)^{2k}.    \\
\end{align*}

In the first sum ($S_1$), we cannot use Weil's bound, so we trivially bound the character sum by $p$, here our goal is to control the cardinality of $\mathcal{J}_1$ to bound the value of the sum $S_1$. We can see that in order 
for $f_\textbf{j}$ to be 
a perfect $d$-th power, it needs to have $l\leq k$ roots (a monomial $(x-j)^{d-1}$ or 
$(x-j')$ cannot alone stand for a $d$-th power). We thus have a trivial injection of 
$\mathcal{J}_1$ into
\begin{align*}
{\mathcal{J}_1}' := \left\{(j_1,\ldots,j_{2k})\in \mathcal{J} \mid \# \{ j_1,\ldots,j_{2k} \} \leq k  \right\}.
\end{align*}

For $\textbf{e} :=(e_1,\ldots, e_l)$ an $l$-tuple of $l$ distinct values, we have at most $2^{2k} l^{2k-l}$ $2k$-tuples $\textbf{j} :=(j_1,\ldots, j_{2k})$ such that the distinct values that appears 
in $\textbf{j}$ are the $\textbf{e}$ in that exact order. Indeed, if we take one such $2k$-tuple, we know that $j_1 = e_1$, then we track the first appearance of $e_i$ for $2 \leq i \leq l$, 
we now have a sequence $1 = i_1 < i_2 < \cdots < i_l \leq 2k$, and by the ``stars and bars'' theorem, there are $\binom{2k-1}{l-1} \leq 2^{2k}$ ways to chose these $l-1$ integers 
$i_2,\ldots, i_l$. The $2k-l$ remaining values (which are not the first occurence of an $e_i$) can take at most $l$ different values, thus we bound the number of combinations possible 
for them by $l^{2k-l}$ which gives us our result. Since there are repetitions, the modulus of the product of $c_j(x)$ will have the following form:
\[ \prod_{h=1}^{2k} |c_{j_h}(x)| = \prod_{h=1}^{l} |c_{e_h}(x)|^{2+\varepsilon_{h}},
\]
with $2+\varepsilon_{h}$ being the number of occurence of $e_h$ in $\textbf{j}$. Moreover, $\sum_{h=1}^l (2 +\varepsilon_{h}) = 2k$,
and when we sum this product over the set of possible values of $(j_h)_{1\leq h \leq l}$, we get

\[
\sum_{M < |j_1|,\ldots,|j_l|<N } \prod_{h=1}^l |c_{j_h}(x)|^{2+\varepsilon_h} \leq C_0^{2k}\prod_{h=1}^l \frac{2}{M^{1+\varepsilon_{h}}} =  \frac{C_0^{2k}2^l}{M^{2k-l}}.
\]

We will now split the sum along the possible sets of such exponents which will give us the desired bound on $S_1$:

\begin{align*}
S_1 &= \dfrac{1}{p} \sum_{\textbf{j}\in \mathcal{J}_1}  \prod_{h=1}^k c_{j_h}(x)\overline{c_{j_{h+k}}(x)}  \sum_{K=1}^p\chi\left(\prod_{h=1}^k (K-j_h)^{d-1}\prod_{h=k+1}^{2k}(K-j_h)  \right) \\
&\leq  \sum_{\textbf{j}\in \mathcal{J}_1'}  \prod_{h=1}^{2k} |c_{j_h}(x)|  \\
&\leq \sum_{l=1}^k \sum_{\substack{(\varepsilon_1,\ldots,\varepsilon_l) \\ \sum_{1\leq h \leq l}\varepsilon_h =2(k-l)  }} 2^{2k} l^{2k-l}\sum_{M < |e_1|,\ldots,|e_l|<N } 
\prod_{h=1}^l |c_{e_h}(x)|^{2+\varepsilon_h} \\
&\leq \sum_{l=1}^k \binom{2k-l-1}{l-1} 2^{2k+l}C_0^{2k} \left(\dfrac{l}{M} \right)^{2k-l} \\
&\leq (16C_0^2)^{k}  \sum_{l=1}^k \left(\dfrac{l}{M} \right)^{2k-l} .
\end{align*}
\noindent
Moreover, if $M\geq \frac{k}{e}$, we can see that $x\mapsto \left(\dfrac{x}{M} \right)^{2k-x} $ is increasing on $[0,k]$, thus \newline $\sum_{l=1}^k \left(\dfrac{l}{M} \right)^{2k-l}  \leq k\left(\dfrac{k}{M} \right)^{k} \leq \left(\dfrac{2k}{M} \right)^{k}  $ which completes the proof.

\end{proof}

\medskip

\begin{coro}
    \label{coro:MajorationSommeZMN-CasParticulier}

 Let $p$ be a large prime and $1\leq M < N < \frac{p}{2}$ be positive integers. Then, for all positive integers $k\leq M$, we have
    
\begin{equation}
    \frac{1}{p} \sum_{K\in\F_p} \left| \sum_{M < |j| < N}  \chi(K-j)\dfrac{1}{p(e_p(j+x) - 1)} \right|^{2k}   \leq  \left(\dfrac{8k}{M} \right)^{k}  + \dfrac{2k}{\sqrt{p}} \left(\frac{1}{2}\log\frac{N}{M} \right)^k.
\end{equation}
\end{coro}

\begin{proof}

    For $-p/2 < j < p/2,$ $|j|\geq 2$,
    \[  \frac{1}{|p(e_p(j+x)  - 1)|} \leq \frac{1}{2(|j|-1)}.
    \]

\vspace{0.5em}

    We apply Proposition~\ref{prop:MajorationSommeMN} with $c_j(x) = \frac{1}{p(e_p(j+x))}$, which is bounded by $\frac{1}{2(|j|-1)}$.
\end{proof}

\medskip

With these tools, we can now prove the following proposition, which will allow us to bound the moments of the second sum in~\eqref{eq:gpkSplitSum}.
\begin{prop}
\label{prop:Section4MajMoments}
    Let $\{c_j\}_{j\in \Z}$ be a sequence of complex-valued functions of a variable $x\in [0,1]$ such that $|c_j(x)| \leq \frac{C_0}{|j|-1}$ and $ \underset{x\in[0,1]}{\sup} |c_j'(x)| \leq \frac{C_0}{(|j|-1)^2}$ for $|j| \geq 2$, where $C_0$ is a positive constant. Let $p$ be a large prime, $k\in\N, k\leq \frac{\log p}{100\log_2 p}$ and put $J = 10^5 C_0^2 k$. Then we have

\[
    \mathcal{M}_{2k} :=\frac{1}{p}\sum_{K\in\F_p} \max_{x \in [0,1]} \left| \sum_{J< |j| <\frac{p}{2}} \chi(K-j)c_j(x)\right|^{2k} \ll e^{-7k}.
\]
\end{prop}

\begin{proof}

Let $S_k := \{\frac{l}{k^4} : 0\leq l \leq k^4 \}$. Then, for all $x\in [0,1]$, there exists $x_l\in S_k$ such that $|x - x_l|\leq \frac{1}{k^4}$. In this case, we have $|c_j(x) - c_j(x_l)| \leq \frac{C_0}{(|j|-1)^2 k^4}$, and hence we have
     \[
      \left| \sum_{J < |j| < \frac{p}{2}}  \chi(K-j)c_j(x) \right| \leq \left| \sum_{J < |j| < \frac{p}{2}}  \chi(K-j)c_j(x_l) \right| + \frac{C_0}{(J-1) k^4}.
     \]

     By using this result, Proposition~\ref{prop:MajorationSommeMN}  and the following inequality (valid for $x,y\in\R$): 
     \[|x + y|^{2k} \leq (2\max(x,y))^{2k} \leq 2^{2k}(x^{2k} + y^{2k}),\] we get
     \begin{align}
         \label{eq:MajMomentsSecondSumStep2}
         \begin{split}
            \mathcal{M}_{2k} 
            &\leq\, \frac{2^{2k}}{p} \sum_{K\in\F_p} \left( \max_{x \in S_k} \left| \sum_{J < |j| < \frac{p}{2}}  \chi(K-j)c_j(x) \right|^{2k} + \left(\frac{C_0}{(J-1) k^4}\right)^{2k} \right) \\
            &\leq\, \frac{2^{2k}}{p} \sum_{K\in\F_p} \left( \sum_{x \in S_k} \left| \sum_{ J < |j| < \frac{p}{2}}  \chi(K-j)c_j(x) \right|^{2k} + \left(\frac{C_0}{(J-1) k^4}\right)^{2k} \right)  \\
            &\leq\, 2^{2k} k^4  \left( \left(\dfrac{32C_0^2k}{J} \right)^{k}  + \dfrac{2k}{\sqrt{p}} \left(2C_0\log p \right)^{2k}+ \left(\frac{C_0}{(J-1) k^4}\right)^{2k} \right)  \\
            &\ll\, e^{-7k}.
         \end{split}
    \end{align}
\end{proof}

\medskip

With this proposition, we can now prove the upper bound of Theorems~\ref{theorem:distributionEvenOrderCharacterSum} and~\ref{theorem:distributionOddOrderCharacterSum}. 
In the following, we recall that $\delta_d = 2$ if $d$ is even, and $\delta_d = 2\cos\frac{\pi}{2d}$ if $d$ is odd.

\begin{proof}[Proof of the upper bound of Theorems~\ref{theorem:distributionEvenOrderCharacterSum} and~\ref{theorem:distributionOddOrderCharacterSum}]

Let
\[M_\chi(K):= \max_{x\in[0,1]} \left|(e(x) - 1)g_{\chi,K}(x) \right| = \max_{x\in[0,1]} \left|\frac{f_\chi(\zeta_p^{K+x})}{\sqrt{p}}\right|.\]
Let also $c_j(x):= \dfrac{2\sin(\pi x)}{p(e_p(j+x)-1)}$. It is easy to see that $|c_j(x)| \leq \dfrac{1}{|j|-1}$ and $\sup_{x\in [0,1]} |c_j'(x)| \leq \dfrac{1}{(|j|-1)^2}$ for $-p/2 < j < p/2$.
 Let $k\leq \dfrac{\log p}{100\log_2 p}$ be a positive integer that we will choose later and let $J = 10^5 k$. It follows from Lemmas~\ref{lemma:HarmonicEstimation} and~\ref{lemma:maxOfSumUnitCircle} that

\begin{align*}
M_\chi(K) 
&\leq  \max_{x\in [0,1]} \left|\sum_{0\leq |j|< J} \frac{\chi(K-j)}{\pi(j+x)}\right| + \max_{x\in [0,1]} \left|\sum_{J\leq |j|< p/2} \chi(K-j)c_j(x)\right| \\ 
&\leq \frac{\delta_d}{\pi}\log k + \max_{x\in [0,1]} \left|\sum_{J\leq |j|< p/2} \chi(K-j)c_j(x)\right| + C_1,
\end{align*}
\noindent
with $C_1 = \frac{\delta_d}{\pi}\left(\gamma + 2\log 2 + 5\log 10 \right) $. 
\medskip

Let $C_d^+ = 7\exp(-\frac{\pi}{\delta_d}(C_1+1))$. We now choose $k = \frac{C_d^+}{7}\exp(\frac{\pi}{\delta_d}V)$. 
Then, by Proposition~\ref{prop:Section4MajMoments}, since $|c_j(x)| \leq \frac{C_0}{|j|-1}, |c_j'(x)| \leq \frac{C_0}{(|j|-1)^2}$ with $C_0 = 1$, we have

\begin{align*}
         \begin{split}
            \Phi_\chi(V) 
            &\leq \frac{1}{p} \left| \left\{K\in\F_p :\; \max_{x\in [0,1]}\left|\sum_{J\leq |j|< p/2} \chi(K-j)c_j(x)\right| \geq 1 \right \} \right| \\
            &\leq \frac{1}{p} \sum_{K\in\F_p} \max_{x \in [0,1]} \left| \sum_{J \leq |j| < \frac{p}{2}}  \chi(K-j)c_j(x) \right|^{2k}   \\
            &\ll e^{-7k}  \\
            &\ll \exp\left(-C_d^+\exp\left( \frac{\pi}{\delta_d}V \right) \right)  .
         \end{split}
    \end{align*}
\end{proof}

\medskip

\section{The random model and preliminary results for the saddle point method}
\label{section4:Part1LowerBound}

In order to provide a lower bound for $\Phi_\chi$, we will compare the Laplace transform of $g_{\chi,K}$ to its probabilistic counterpart, which we define below: 
\[ 
G_{\mathbb{X},\chi}(x) := \frac{i}{p}\sum_{|j| < p/2} \dfrac{\mathbb{X}(j)}{e_p(j+x)-1},
\]
where $(\X(j))$ are i.i.d uniform random variables on $\mathbb{U}_d$, $d$ being the order of $\chi$ and $p$ its modulus.

\medskip

Throughout this section, we adopt a lighter notation for the sake of readability, writing $G_{\X,\chi} := G_{\X,\chi}(1/2)$ and similarly $g_{\chi,K}:= g_{\chi,K}(1/2) $.

\begin{prop}
    \label{prop:linkgpkAndGX}
    There exists a set $\mathcal{E}_p \subset \F_p$, of cardinality $|\mathcal{E}_p| \leq p^{3/4}$, such that for all $s\in \C, |s| \leq \dfrac{\log p}{100(\log_2 p)^2}$, we have

    \[
    \frac{1}{p}\sum_{K\in \F_p \setminus \mathcal{E}_p} \exp{\!\left(2s  \Re (g_{\chi,K}) \right)} = \E\left(\exp{\!\left(2s \Re (G_{\X,\chi})  \right)}  \right)+ O\!\left( \exp{\!\left( - \frac{\log p }{60\log_2 p} \right)}   \right),
    \]
    and the implicit constant is absolute.
\end{prop}

In order to prove this proposition, we will need a few tools.

\begin{lemma}
\label{lemma:productofArithmeticToProbabilistic}
For $k_1,k_2 \in \N$, if $\chi \pmod p$ is a character of order $d$, the following holds

    \[
    \frac{1}{p}\sum_{K\in \F_p} g_{\chi,K}^{k_1} \overline{g_{\chi,K}}^{k_2}  = \E\Big(G_{\X,\chi}^{k_1} \overline{G_{\X,\chi}}^{k_2}  \Big)+ O\!\left( \frac{k_1+k_2}{\sqrt{p}}\left(\frac{1}{\pi}\log p\right)^{k_1+k_2}  \right).
    \]

\end{lemma}

To derive this expression,  we take inspiration from \cite[Proposition 4.1]{klurman2023lqnormsmahlermeasure}.

\begin{proof}[Proof of Lemma~\ref{lemma:productofArithmeticToProbabilistic}]

First, we note that if $\X$ is a random variable uniformly distributed over the $d$-th roots of unity, 
then $\E(\X^k) = \mathds{1}(d \mid k)$, so we deduce that for $(j_1,\ldots, j_{k_1+k_2})$ a ($k_1+k_2$)-tuple of integers,

\[ \E\left( \prod_{1 \leq h \leq k_1} \X(j_h)\prod_{k_1 +1\leq h \leq k_1+k_2} \X(j_{h})^{d-1}  \right) = \mathds{1}(f_{\textbf{j},k_1} \text{ is a perfect } d\text{-th power}),  \]

\noindent
where $f_{\textbf{j},k_1}$ is defined in ~\eqref{eq:def_polynomeTuple}. Moreover, by Weil's bound (Lemma~\ref{WeilBound}), we have a similar estimate for the character sum:
\[ 
\frac{1}{p}  \sum_{K=1}^p \chi(f_{\textbf{j},k_1}(K)) = \mathds{1}(f_{\textbf{j},k_1} \text{ is a perfect } d\text{-th power}) + O\!\left(\frac{k_1+k_2}{\sqrt{p}}\right) .
\]
Thus, the two quantities are equal up to an error $ O\!\left(\frac{k_1+k_2}{\sqrt{p}}\right)$.

    Define $c_j := \dfrac{i}{p\left(e_p(j+1/2) -1\right)}$. 
    As seen in~\eqref{eq:approxCoeffSerieFourier}, we have uniformly for all $|j|< p/2$,
    \begin{equation*}
    c_j = \dfrac{1}{\pi(2j+1)} + O\!\left( \frac{1}{p} \right). 
    \end{equation*}

\noindent
    Moreover, since $\chi(j)$ and $\X(j)$ are in $\U_d$, $\overline{\chi}(j) = \chi(j)^{d-1}$ and $\overline{\X}(j) = \X(j)^{d-1}$. This gives us the following:

    \begin{align*}
        & \frac{1}{p} \sum_{K\in \F_p}g_{\chi,K}^{k_1} \overline{g_{\chi,K}}^{k_2}   \\
        =\; & \frac{1}{p}\sum_{\substack{0 \leq |j_h|  < p/2, \\ 1\leq h \leq k_1+k_2}} \;\prod_{1\leq h \leq k_1}c_{j_h}\prod_{k_1 < h \leq k_1+k_2} \overline{c_{j_{h}} } \sum_{K\in\F_p} 
         \chi(f_{\textbf{j},k_1}(K)) \\
        =\; & \sum_{\substack{0 \leq |j_h|  < p/2, \\ 1\leq h \leq k_1+k_2}} \;\prod_{1\leq h \leq k_1}c_{j_h}\prod_{k_1< h \leq k_1+k_2}  \overline{c_{j_{h}} }  
        \;\E\left( \prod_{1 \leq h \leq k_1} \X(j_h) \prod_{k_1< h \leq k_1+k_2}  \X(j_{h})^{d-1} \right) + E_{k_1,k_2} \\
        =\; &  \E\Big(G_{\X,\chi}^{k_1} \overline{G_{\X,\chi}}^{k_2}  \Big) + E_{k_1,k_2},
    \end{align*}
\noindent
    where the error term $E_{k_1,k_2}$ can be bounded as follows:
    \[ E_{k_1,k_2} \ll \frac{k_1+k_2}{\sqrt{p}}\sum_{\substack{0 \leq |j_h|  < p/2, \\ 1\leq h \leq k_1+k_2}} \; \prod_{1\leq h \leq k_1+k_2} |c_{j_h}| \ll \frac{k_1+k_2}{\sqrt{p}} \left(\frac{1}{\pi}\log p \right)^{k_1+k_2},
    \]
    and the implicit constant is absolute.
      
\end{proof}

\begin{prop}
\label{prop:momentArithmetic=Probabilistic}
For $n\in \N$, the following holds\textnormal{:} 

\[
\dfrac{1}{p}\sum_{K\in\F_p} \big( \Re ( g_{\chi,K}) \big)^n = \E\Big(\big( \Re(  G_{\X,\chi})\big) ^n \Big) + O\!\left( \frac{n}{\sqrt{p}}\left(\frac{1}{\pi} \log p\right)^n \right),
\]
\noindent
and the implicit constant is absolute.
\end{prop}

\medskip
\begin{proof}[Proof of Proposition~\ref{prop:momentArithmetic=Probabilistic}]

By Lemma~\ref{lemma:productofArithmeticToProbabilistic}, we have

\begin{align*}
\dfrac{1}{p}\sum_{K\in\F_p} \Re( g_{\chi,K} )^n &= \dfrac{1}{p 2^n}\sum_{K\in\F_p} \sum_{k=0}^n \binom{n}{k} g_{\chi,K}^k \overline{g_{\chi,K}}^{n-k} \\
&= \dfrac{1}{ 2^n} \sum_{k=0}^n \binom{n}{k} \left( \E\Big(G_{\X,\chi}^{k} \overline{G_{\X,\chi}}^{n-k}  \Big) + E_{k,n-k} \right) \\
&= \E\left( \Re ( G_{\X,\chi})^n \right) + \dfrac{1}{ 2^n} \sum_{k=0}^n \binom{n}{k} E_{k,n-k}.
\end{align*}
\noindent
Moreover,

\begin{equation*}
\left|\dfrac{1}{ 2^n } \sum_{k=0}^n \binom{n}{k} E_{k,n-k} \right| 
\ll \dfrac{1}{ 2^n } \sum_{k=0}^n \binom{n}{k} \frac{n}{\sqrt{p}}\left(\frac{1}{\pi} \log p\right)^n 
 \ll \frac{n}{\sqrt{p}}\left(\frac{1}{\pi} \log p\right)^n.
\end{equation*}

\end{proof}

Now, we move on to the proof of Proposition~\ref{prop:linkgpkAndGX}.

\begin{proof}[Proof of Proposition~\ref{prop:linkgpkAndGX}]

Let $\mathcal{E}_p$ be the set of $K\in\F_p$ such that
\[
|g_{\chi,K}| \geq  \log_2{p} .
\]

By Lemma~\ref{lemma:HarmonicEstimation}, we get for $p$ sufficiently large

    \begin{align*}
        |g_{\chi,K}|
        &\leq \sum_{ 1 \leq |j|< (\log p)^2} \frac{1}{\pi (2j+1)} + \left| \sum_{ (\log p)^2\leq |j| < p/2}  \chi(K-j)\dfrac{1}{p(e_p(j+\frac{1}{2}) - 1)} \right| \\
        &\leq  \frac{7}{10}\log_2 p + \left| \sum_{ (\log p)^2\leq |j| < p/2}  \chi(K-j)\dfrac{1}{p(e_p(j+\frac{1}{2}) - 1)} \right| .\\
    \end{align*}

Using Corollary~\ref{coro:MajorationSommeZMN-CasParticulier} with $k = \left \lfloor \dfrac{\log p}{4 \log_2 p} \right \rfloor$ we obtain
    
    \begin{align*} 
        |\mathcal{E}_p| &\leq  \left| \left\{  K \in \F_p: \; \left| \sum_{ (\log p)^2\leq |j| < p/2}  \chi(K-j)\dfrac{1}{p(e_p(j+\frac{1}{2}) - 1)} \right| \geq \frac{3}{10}\log_2 p  \right\}  \right|  \nonumber \\
        &\leq \frac{10^{2k}}{(3\log_2 p)^{2k}} \sum_{K\in\F_p} \left| \sum_{ (\log p)^2\leq |j| < p/2} \dfrac{ \chi(K-j,p)}{p(e_p(j+\frac{1}{2}) - 1)} \right|^{2k}  \\
        &\ll p^{3/4} . \nonumber
    \end{align*}

We now outline the derivation of Proposition~\ref{prop:linkgpkAndGX}.
\begin{itemize}
    \item First, we write the exponential on the left hand side as its power series, and we will put its tail in a first error term, $\kappa_1$.
    \item Then, we complete the sum by adding back the excluded exceptional set $K\in \mathcal{E}_p$, and the quantity that we add will be $\kappa_2$.
    \item Then, we will replace the first moments of the arithmetic quantity by its probabilistic counterpart, and their difference will be $\kappa_3$.
    \item Then, we will complete the head of the power series of $\exp(2\Re( G_{\X,\chi}))$, and the tail that we add will be written as $\kappa_4$.
\end{itemize}

\medskip

    We shall cut the power series at the parameter $N := \frac{\log p}{12\log_2 p}$. The following inequality $2e^{1.2}|s|\log_2 p  \leq N$ holds, so since for $K\in \F_p \setminus \mathcal{E}_p, |g_{\chi,K}|<\log_2 p$, we have

    \begin{align}
        \label{eq:BoundError1}
        \begin{split}
    \kappa_1 
    = \left| \frac{1}{p}\sum_{K\in \F_p \setminus \mathcal{E}_p} \sum_{n=N+1}^{\infty}  \dfrac{s^n}{n!}\left( 2g_{\chi,K} \right)^n \right|  
    &\leq \sum_{n=N+1}^{\infty}  \dfrac{(2|s|\log_2 p)^n}{n!} \\
    &\leq \sum_{n=N+1}^{\infty}  \dfrac{(2e|s|\log_2 p)^n}{N^n} \\
    &\ll  \left( \dfrac{2e|s|\log_2 p}{N} \right)^{N}   \\
    &\ll e^{-0.2N} . 
            \end{split}
    \end{align}

For the second error term, we use the fact that $g_{\chi,K} \leq \log p$ to get

    \begin{align}
        \label{eq:BoundError2}
        \begin{split}
        \kappa_2 
    = \left| \frac{1}{p}\sum_{K\in  \mathcal{E}_p} \sum_{n=0}^{N} \dfrac{(2s)^n}{n!} \Re ( g_{\chi,K} ) ^n  \right|  
    &\leq  \frac{1}{p^{1/4}} \sum_{n=0}^{N} \dfrac{(2|s|\log p)^n}{n!}  \\
    &\ll  \frac{1}{p^{1/4}} (2|s|\log p)^N   \\
    &\ll p^{-1/12} .         
        \end{split}
    \end{align}

For the third error term, by using Proposition~\ref{prop:momentArithmetic=Probabilistic} we have

    \begin{align}
        \label{eq:BoundError3}
    \kappa_3 \ll \sum_{n=0}^{N} \dfrac{|s|^n n\left(\frac{2}{\pi}\log p\right)^n}{\sqrt{p}n!} \ll \dfrac{e}{\sqrt{p}} \left(\dfrac{2|s|\log p}{\pi } \right)^{N} \ll p^{-\frac{1}{6}}.
    \end{align}

    For the fourth error, we need to bound:
    \[
    \kappa_4 = \left| \sum_{n=N+1}^{\infty}  \dfrac{2^ns^n}{n!}\E \left( \Re(  G_{\X,\chi})^n \right)\right|.
    \]
\noindent
    We first note that for $ |j| < p/2,$
    \[  |c_j| = \frac{1}{|p(e_p(j+\frac{1}{2})  - 1)|} \leq \frac{1}{2 |j+\frac{1}{2}|  }.
    \] 
   
\noindent
    Therefore, we have for $n \in \N$, from Lemma~\ref{lemma:HarmonicEstimation}:
    \begin{align}
        \label{eq:BoundError4-1}
        \begin{split}
        \E \left( \left|\sum_{|j| < 4n}  \Re (c_j\mathbb{X}(j))\right|^{n} \right) 
        &\leq \E \left( \left(\sum_{|j| < 4n} \dfrac{1}{2|j+\frac{1}{2}|}\right)^{n} \right) \\ 
        &\leq \left(\log n + \gamma + 4\log 2 \right)^{n}.
        \end{split}
    \end{align}

\noindent
    Moreover, we have the following:

    \begin{align}
        \label{eqstep:BoundError4-2}
       \E \left(\left(\sum_{4n\leq |j| < \frac{p}{2}} c_j \Re( \mathbb{X}(j))\right)^{n} \right)
       &= \sum_{\substack{4n \leq |j_h|  < p/2, \\ 1\leq h \leq n}} \left( \prod_{h=1} ^{n} c_{j_h} \right) \E\left( \prod_{h=1}^n \dfrac{\X_{j_h}  +  \X_{j_h}^{d-1}}{2} \right).
    \end{align}    

\noindent
    Expanding the product on the right hand side gives us

    \[
    \E\left( \prod_{h=1}^n \dfrac{\X_{j_h}  +  \X_{j_h}^{d-1}}{2} \right) = \frac{1}{2^n}\sum_{I\subset \{1,\ldots,n\}}  \E\left( \prod_{h\notin I}\X_{j_h} \prod_{h\in I} \X_{j_h}^{d-1} \right).
    \]

\noindent
    The same combinatorial argument as in Proposition~\ref{prop:MajorationSommeMN} can be made since as seen in the 
    proof of Lemma~\ref{lemma:productofArithmeticToProbabilistic}, $\E(\X(j_1)^{\alpha_1}\cdots \X(j_n)^{\alpha_n}) \neq 0$ 
    if and only if \( (X-j_1)^{\alpha_1}\cdots (X-j_n)^{\alpha_n}\) is a perfect $d$-th power. 
    We can now bound more properly the quantity in (\ref{eqstep:BoundError4-2}) by the first term in 
    Corollary~\ref{coro:MajorationSommeZMN-CasParticulier} (since we have perfect orthogonality in this case). We get

    \begin{equation}
        \label{eq:BoundError4-2}
         \E \left(\left|\sum_{4n\leq |j| < \frac{p}{2}} c_j \Re( \mathbb{X}(j))\right|^{n} \right) 
       \leq   \left(\dfrac{8 n/2}{ 4n} \right)^{n/2} \ll 1.
    \end{equation}

\noindent
    
    By \eqref{eq:BoundError4-1}, \eqref{eq:BoundError4-2} and Minkowski's inequality we then have the following:

    \begin{align*}
        \E \left( \left|\Re( G_{\X,\chi} )\right|^{n} \right)^{\frac{1}{n}} 
        &\leq   \E \left( \left|\sum_{|j| < 4n} c_j \Re ( \X(j)) \right|^{n} \right)^{\frac{1}{n}} + \E \left( \left|\sum_{4n\leq |j| < \frac{p}{2}} c_j \Re(\X(j)) \right|^{n} \right)^{\frac{1}{n}}  \\
        &\leq \log n +5,  \\
    \end{align*}

\noindent
    which gives us
    \begin{align}
        \E \left( \left|\Re (G_{\X,\chi}) \right|^{n} \right)  \ll  ( \log n +5)^{n}.
    \end{align}

\noindent
    Thus, due to the decreasing nature of $x\mapsto \frac{\log x}{x}$ and the use of Stirling's formula, we deduce the following bound
    for the third error term:

    \begin{align} \label{eq:BoundError4}
        \begin{split}
    \kappa_4 \leq  \sum_{n\geq N+1} \dfrac{(|s| ( \log n + 5 ))^{n}}{n!} &\leq  \sum_{n\geq N+1} \left( \dfrac{2e|s| \log n}{n} \right)^{n}  \\
    &\leq \sum_{n\geq N+1} \left( \dfrac{2e|s| \log N}{N} \right)^{n} \\
    &\ll e^{-0.2 N}. 
        \end{split}
    \end{align}

    \medskip
    Combining \eqref{eq:BoundError1}, \eqref{eq:BoundError2}, \eqref{eq:BoundError3} and \eqref{eq:BoundError4}  gives us the following result:
    
    \begin{align*}
     \frac{1}{p}\sum_{K\in \F_p \setminus \mathcal{E}_p} \exp{\!\left(2s \Re ( g_{\chi,K} )  \right)} 
     = \;&  \frac{1}{p}\sum_{K\in \F_p \setminus \mathcal{E}_p} \sum_{n=0}^{N}  \dfrac{(2s)^n}{n!} \Re ( g_{\chi,K} ) ^n + O(\kappa_1)  \\
     = \;&  \frac{1}{p}\sum_{K\in \F_p } \sum_{n=0}^{N}  \dfrac{(2s)^n}{n!} \Re ( g_{\chi,K} ) ^n + O(\kappa_1 + \kappa_2) \\     
     = \;&  \sum_{n=0}^{N} \dfrac{(2s)^n}{n!} \E\left( \Re ( G_{\X,\chi} ) ^n  \right)+ O(\kappa_1 + \kappa_2 + \kappa_3)    \\
     = \;& \E\left(\exp{\!\left(2s\Re ( G_{\X,\chi} )  \right)}  \right)+ O(\kappa_1 + \kappa_2 + \kappa_3 + \kappa_4)  \\
     = \;& \E\left(\exp{\!\left(2s \Re ( G_{\X,\chi} )  \right)}  \right) + O( e^{-0.2 N} ).
    \end{align*}

\end{proof}

\medskip

\section{Proof of the lower bound}
\label{section5:proofLowerBound}

With the link between the arithmetic and probabilistic models established, we proceed to 
the saddle-point method, drawing inspiration from~\cite{LamzouriMaxCubicSum}, which will allow us to 
derive Theorem~\ref{theorem:distribution1/2CharacterSum} and the lower bound of Theorem~\ref{theorem:distributionEvenOrderCharacterSum}, 
 which we state precisely below.

 \begin{theorem}
    \label{thm:LowerBoundEvenThm}
 Let $p$ be a large prime. Uniformly for all Dirichlet characters $\chi\pmod p$ of even order $d$ and
for all real numbers $ 1\leq V \leq \frac{2}{\pi}(\log_2 p - 2\log_3 p)$, we have

\begin{equation*}
    \exp\left( - C_d^- \exp \left(\frac{\pi}{2}V  \right)  \left(  1 + O(e^{-\pi V/4 }) \right)   \right) \leq \Phi_\chi(V),
\end{equation*}
where in the last inequality,
\[ C_d^- = \frac{2}{\pi}\exp\left( -\gamma - 1 + \log \frac{\pi}{2} - \displaystyle\int_0^1 \frac{\alpha_d(u)}{u^2}\, du 
- \displaystyle\int_1^{\infty} \frac{\alpha_d(u) - u}{u^2}\, du \right),
\]
where $\gamma$ is the Euler–Mascheroni constant and $\alpha_d$ is defined in~\eqref{eq:def_alpha_d}.

 \end{theorem}

\begin{remark}
    Applying the same method as in the proof of Theorem~\ref{thm:LowerBoundEvenThm}, one can obtain 
    an analogous lower bound in the case of odd order characters. However, taking the real part of the $g_{\chi,K}$ instead of
    its absolute value leads to a decay scale of the form $\exp\left(-\exp(\frac{\pi}{1+\cos\frac{\pi}{d}}V)\right)$ rather than the sharper factor $\frac{1}{2\cos\frac{\pi}{2d}}$
    that we have in Theorem~\ref{theorem:distributionOddOrderCharacterSum}. 
    As a consequence, this approach is insufficient to capture the correct behavior for small values of the order \(d\). 
    Therefore, in Section~\ref{section6:proofoddLowerBound}, we exploit the effective independence of character values on short intervals to obtain the desired result, yielding the optimal dependence on \(d\).

    Nevertheless, this approach has the advantage of producing fully explicit constants, and in certain regimes of the parameter \(d\), the 
    resulting bounds are sharp. In particular, when \(d \gg \exp(V)\), one recovers the same results as Theorem~\ref{theorem:distributionOddOrderCharacterSum}
     with better constants.  

\end{remark}

\vspace{0.5em}Take $s\in \R, 2 \leq s  \leq \dfrac{\log p}{100(\log_2 p)^2}$, and define
\[
\tilde{\Phi}_\chi(V) = \dfrac{1}{p} \left|\left\{K \in \F_p\setminus \mathcal{E}_p : \; 2\Re(g_{\chi,K}) \geq V \right\} \right|,
\]
where $\mathcal{E}_p$ is defined in the proof of Proposition~\ref{prop:linkgpkAndGX}. 
The quantity defined above is a lower bound for $\Phi_\chi(V)$ because we have the following
(see Definition~\ref{def:g-auxiliary} for the link between $f_\chi$ and $g_{p,K}$):
\[\max_{x\in[0,1]} \left|\frac{f_\chi(e_p(K+x))}{\sqrt{p}}\right|  \geq \left|\frac{f_\chi(e_p(K+\frac{1}{2})}{\sqrt{p}}\right| = 2\left|g_{\chi,K}\left(\frac{1}{2}\right)\right| \geq 2\Re(g_{\chi,K}).
\]
\noindent
Proposition~\ref{prop:linkgpkAndGX} gives the following:

\begin{align*}
    \int_{-\infty}^{\infty} e^{st}\tilde{\Phi}_\chi(t)\,dt 
    &= \dfrac{1}{p} \sum_{K\in\F_p\setminus \mathcal{E}_p} \int_{-\infty}^{ 2\Re(g_{\chi,K})} e^{st}dt  \\
    &= \dfrac{1}{p} \sum_{K\in\F_p\setminus \mathcal{E}_p}  \exp \left(  2 s \Re(g_{\chi,K}) \right) \\
    &= \E \left( \exp \left( 2s \Re(G_{\X,\chi}) \right) \right)  + O\!\left( \exp{\!\left( - \frac{\log p }{60\log_2 p} \right)} \right).
\end{align*}

By definition of $G_{\X,\chi}$ and the approximation of its coefficients~\eqref{eq:approxCoeffSerieFourier}, we have
\begin{align*}
\Re(G_{\X,\chi}) =\; &\frac{1}{\pi}\sum_{j\leq (\log p)^2} \frac{\Re(\X(j))}{2j+1} + O\!\left( \frac{(\log p)^2}{p} \right) \\
+ \;&\frac{1}{p}\sum_{ (\log p)^2 < |j| < p/2} \Re\left(\frac{i\X(j)}{e_p(j+\frac{1}{2})-1} \right).
\end{align*}

The $\X(j)$ are independent uniformly distributed on the $d$-th roots of unity, thus, we can write the expected value as follows
\begin{align*} 
\E \left( \exp \left( 2s \Re(G_{\X,\chi}) \right) \right) 
=P_1 \cdot P_2 \cdot \exp\left( O\!\left( \frac{(\log p)^2}{p} \right)  \right),
\end{align*}
where 
\[
P_1 = \prod_{|j|\leq (\log p)^2}\left(\frac{1}{d}\sum_{k=0}^{d-1} \exp\left(\frac{2s\cos\left( \frac{2k\pi}{d} \right)}{\pi(2j+1)}\right)\right)   \\
\]
gives the main contribution and
\[
P_2 =  \prod_{(\log p)^2 < |j| \leq p/2} \left(\frac{1}{d}\sum_{k=0}^{d-1}\exp\left(\frac{s}{p}\left(\cos \frac{2k\pi}{d} \cot\left( \frac{\pi(2j+1)}{2p} \right) + \sin \frac{2k\pi}{d} \right)   \right)  \right) \nonumber \\
\]
is a negligible term. In order to estimate $P_1$ and $P_2$, we will study their logarithms.

\begin{prop} \label{prop:estimateLogEsperance}
     For $s\in \C, |s| \leq \dfrac{\log p}{\log_3 p}$, the following holds for even $d$\textnormal{:}

      \[
      \log\E \left( \exp \left( 2s \Re(G_{\X,\chi}) \right) \right)  = \frac{2s}{\pi}\log s + C_d s+ O(\log s),
      \]
    \noindent
    with $C_d := \hat{C}_d + \frac{2}{\pi}\int_{0}^1 \frac{\alpha_d(u)}{u^2}\,du$ where $\hat{C}_d$ is defined in Lemma~\ref{lemma:HeadailOfEstimate} and $\alpha_d$ in~\eqref{eq:def_alpha_d}.
\end{prop}

To prove Proposition~\ref{prop:estimateLogEsperance}, we shall first prove in the next lemma that $P_2$ is indeed a negligible term.
\begin{lemma} \label{lemma:TailOfEstimate}
    For $s\in \C, |s| \leq \dfrac{\log p}{\log_3 p}$, the following holds:
    \[
    \sum_{(\log p)^2 < |j|\leq p/2}\log\!\left(\frac{1}{d}\sum_{k=0}^{d-1}\exp\left(\frac{s}{p}\left(\cos \frac{2k\pi}{d} \cot\left( \frac{\pi(2j+1)}{2p} \right) + \sin \frac{2k\pi}{d} \right)   \right)  \right) = O\!\left(\frac{1}{(\log_3 p)^2}\right).
    \]
\end{lemma}

\begin{proof}[Proof of Lemma~\ref{lemma:TailOfEstimate}]
    For  $x\in \left[ -\frac{\pi}{2}, \frac{\pi}{2}\right]$, $0 \leq  x \cot(x) \leq 1 $, thus we have the following for $(\log p)^2 < |j| < p/2$ and $0\leq k \leq d-1$,

    \begin{align*}
        &\frac{s}{p}\left(\cos\left( 2k\pi/d \right)\cot\left( \frac{\pi(2j+1)}{2p} \right) + \sin\left( 2k\pi/d \right) \right) \\
        =\;& \frac{2s}{\pi( 2j+1)}\left(\frac{\pi(2j+1)}{2p} \cot\left( \frac{\pi(2j+1)}{2p} \right) \cos\left( 2k\pi/d \right) + \frac{\pi(2j+1)}{2p}\sin\left( 2k\pi/d \right) \right) \\
        =\;& O\!\left( \frac{|s|}{j}\right) = o(1).
    \end{align*}
   
    If we write $\alpha_{j,k} := \frac{\pi(2j+1)}{2p} \cot\left( \frac{\pi(2j+1)}{2p} \right) \cos\left( 2k\pi/d \right) + \frac{\pi(2j+1)}{2p}\sin\left( 2k\pi/d \right)$, since $\sum_{k=0}^{d-1}\cos\left( 2k\pi/d \right) = \sum_{k=0}^{d-1}\sin\left( 2k\pi/d \right) = 0 $, we have

    \begin{align*}
    & \frac{1}{d}\sum_{k=0}^{d-1}\exp\left(\frac{s}{p}\left(\cos\left( 2k\pi/d \right)\cot\left( \frac{\pi(2j+1)}{2p} \right) + \sin\left( 2k\pi/d \right) \right)   \right) \\
    =\;& 1 + \frac{2s}{\pi d(2j+1)}\sum_{k=0}^{d-1} \alpha_{j,k} + O\!\left( \frac{|s|^2}{j^2}\right) \\
    =\;& 1 + O\!\left( \frac{|s|^2}{j^2}\right),
    \end{align*}
\noindent
    and the implicit constant in the big $O$ is absolute. Thus
    \begin{align*}
       \sum_{(\log p)^2 < |j|\leq p/2}\log\!\left(\frac{1}{d}\sum_{k=0}^{d-1}\exp\left(\alpha_{j,k}  \right)  \right) 
       &= \sum_{(\log p)^2 < |j|\leq p/2}   O\!\left( \frac{|s|^2}{j^2}\right)   
       =  O\!\left(\frac{1}{(\log_3 p)^2}  \right).
    \end{align*}

\end{proof}

Having established a bound for $P_2$, it remains to analyze $P_1$. In order to do so, 
we will need a few results on the modified Bessel function. For $n\in\Z$, $x\in\R$, let 
\begin{equation} \label{def:I0}
I_n(x):= \displaystyle\sum_{k=0}^{\infty} \dfrac{1}{k!(k+n)!}\left( \dfrac{x}{2}\right)^{2k+n}.     
\end{equation}

$I_n$ is called the modified Bessel function of the first kind. 
It is usually defined as one of the solutions of the following differential equation 
(see \cite{WatsonBesselFct} for a standard reference):
\[
z^2 \dfrac{d^2w}{dz^2} + z\dfrac{dw}{dz} - (z^2+n^2)w = 0.\]

For $x\in\R$, define the function $h_x:\theta \mapsto e^{x\cos \theta}$ for $\theta\in\R$. To estimate the sum in the log on the left hand side of
Lemma~\ref{lemma:HeadailOfEstimate}, we shall use the Fourier expansion of $h_x$ since it is 
$2\pi-$periodic:
\[h_x(\theta) = a_0 + \sum_{n\in\N} a_n \cos(n\theta) + b_n \sin(n\theta),\]
where for $n\geq 1$,
\[
a_n = \frac{1}{\pi}\int_{-\pi}^\pi h_x(\theta)\cos(n\theta) \, d\theta \text{ and } b_n = \frac{1}{\pi}\int_{-\pi}^\pi h_x(\theta)\sin(n\theta) \, d\theta.
\]
Since $h_x$ is even, for $n\in\N,$  $b_n = 0$, $a_n = \dfrac{2}{\pi} \displaystyle \int_0^\pi e^{x\cos \theta}\cos(n\theta) \,d\theta = 2I_n(x)$, 
and $a_0 = \dfrac{1}{\pi} \displaystyle \int_0^\pi e^{x\cos \theta} \,d\theta = I_0(x)$, see for instance [\cite{WatsonBesselFct}, p.181, Formula 4] 
for the first two equalities.
This gives us the following:
\begin{equation}\label{eq:serieAvecCoeffFctBessel}
    h_x(\theta) = I_0(x) + 2\sum_{n=1}^{\infty} I_n(x) \cos(n\theta),
\end{equation}
and if we average $h_x$ over $\{\theta =  \frac{2k\pi}{d} : \; 0\leq k < d \} $, we get
\begin{equation} \label{eq:serieAvecCoeffBessel-unityRoot}
    \frac{1}{d}\sum_{k=0}^{d-1} h_x\left( \frac{2k\pi}{d}\right) = I_0(x) + 2 \sum_{n\geq 1} I_{nd}(x).
\end{equation}
Indeed, averaging $\cos(n\theta)$ over $\{\theta =  \frac{2k\pi}{d} : \; 0\leq k < d \}$ yields $1$ precisely when $d\mid n$, and vanishes otherwise. 
Moreover, it is straightforward from~\eqref{def:I0}
that $I_n$ has the same parity as $n$ and that if $x>0, n \mapsto I_n(x)$ is decreasing and positive.

\medskip

Let $d\geq 2$ an even integer, for $t\in\R$, we define $\alpha_d:\R\rightarrow \R$ as:
\begin{equation} \label{eq:def_alpha_d}
        \alpha_d(u) := \log\left(\frac{1}{d}\sum_{k=0}^{d-1} \exp\left(u\cos\left( \frac{2k\pi}{d} \right)\right)\right).
\end{equation}
\begin{lemma} \label{lemma:estimates_alpha_d}
    Let $d\geq 2$ an even integer. The function $\alpha_d$ satisfies the following estimates\textnormal{:}

\begin{enumerate}
    \item For $u\in \R$,  $0 \leq \log I_0(u) \leq \alpha_d(u) \leq |u| $. \label{eq:property1-alphad}
    \item For $u\in \R$,  $|\alpha_d'(u))| \leq 1$.  \label{eq:property2-alphad}
    \item 
    $\alpha_d(u) \underset{|u|\to 0}{\sim} 
\begin{cases}
 \frac{u^2}{2},\; \text{ if $d=2$,}
 \\
 \frac{u^2}{4},\; \text{ if $d>2$.} 
\end{cases}$ \label{eq:property3-alphad}
    \item 
    $\alpha_d'(u) \underset{|u|\to 0}{\sim} 
\begin{cases}
 u,\; \text{ if $d=2$,}
 \\
 \frac{u}{2},\; \text{ if $d>2$.}
\end{cases}$ \label{eq:property4-alphad}
    \item $\alpha_d$ is an even function, decreasing on $(-\infty,0]$ and increasing on~$[0,\infty)$. \label{eq:property5-alphad}
    \item $u\mapsto \alpha_d(u) - u$ is non-increasing on $\R$. \label{eq:property6-alphad}
\end{enumerate}

\end{lemma}

\begin{proof}[Proof of Lemma~\ref{lemma:estimates_alpha_d}] We treat the above estimates individually.

\begin{enumerate}
    \item  From the definition of $I_0$, the non-negativity of $I_d$ on $\R$ when $d$ is even and 
    equation~\eqref{eq:serieAvecCoeffBessel-unityRoot}, we get for $u\in\R$
    \[
    0\leq \log I_0(u) \leq \log\left(I_0(u) + 2\sum_{n\geq 1}I_{nd}(u) \right) = \alpha_d(u).
    \] 
    The last inequality in~\eqref{eq:property1-alphad} comes from $\exp\left(u\cos\left( \frac{2k\pi}{d} \right)\right) \leq \exp(|u|)$ for $0\leq k < d$.
    
    \item 
    \[
    \alpha_d'(u) = \dfrac{ \sum_{k=0}^{d-1}\cos\left( \frac{2k\pi}{d} \right) \exp\left(u\cos\left( \frac{2k\pi}{d} \right)\right)}{ \sum_{k=0}^{d-1} \exp\left(u\cos\left( \frac{2k\pi}{d} \right)\right)},
    \]
    so taking the absolute value and applying the triangle inequality gives~\eqref{eq:property2-alphad}.
    
    \item Let us recall that $\sum_{k=0}^{d-1}\cos\left(  \frac{2k\pi}{d} \right)=0$ and $\sum_{k=0}^{d-1}\cos^2\left( \frac{2k\pi}{d} \right) = \frac{d}{2}(1+\mathds{1}_{d=2 })$, then
    \[
    \exp(\alpha_d(u)) \underset{u\to 0}{=} 1 + \frac{u}{d}\sum_{k=0}^{d-1}\cos\left( 2k\pi/d \right) + \frac{u^2}{2d}\sum_{k=0}^{d-1}\cos\left( 2k\pi/d \right)^2 + O(u^3).
    \]
    \item 
    The same calculations we did for~\eqref{eq:property3-alphad} on $\alpha_d'$ gives~\eqref{eq:property4-alphad}.

    \item     
    Recall that $\alpha_d(u) = \log\Big(I_0(u) + 2\sum_{n\geq 1}I_{nd}(u)\Big)$. Each $I_{nd}$, $n\geq 0$ is even, decreasing on $(-\infty,0]$ and increasing on~$[0,\infty)$. Combining that with the fact that the logarithm is 
    increasing on $(0,\infty)$ proves~\eqref{eq:property5-alphad}.
    \item 
    Taking the derivative we have computed for~\eqref{eq:property2-alphad} and substracting one gives
    \[
    \dfrac{ \sum_{k=0}^{d-1}\left(\cos\left( \frac{2k\pi}{d} \right)-1\right) \exp\left(u\cos\left( \frac{2k\pi}{d} \right)\right)}{ \sum_{k=0}^{d-1} \exp\left(u\cos\left( \frac{2k\pi}{d} \right)\right)} \leq 0,
    \]
    since each term in the numerator is non-positive, and the denominator is positive. That proves~\eqref{eq:property6-alphad}.
\end{enumerate}

\end{proof}

We require information on the function $I_0$, which we state in the following lemma.

\begin{lemma} \label{lemma:boundsI0}
    For all $u\ge 1$, one has the bounds
     \[
    \dfrac{e^u}{2\pi  \sqrt{u}}  \leq I_0(u) \leq e^u.
    \]  
    Moreover, $u\mapsto e^{-u}I_0(u)$ is decreasing on $(0,\infty)$.

\end{lemma}

\begin{proof}[Proof of Lemma~\ref{lemma:boundsI0}]
    The upper bound comes from~\eqref{eq:property1-alphad} of Lemma~\ref{lemma:estimates_alpha_d}.
    For the lower bound, we recall the integral form of $I_0$, namely for $u\geq 0$,
    \[
    I_0(u) = \int_0^1 e^{u\cos\pi \theta} \, d\theta ,
    \]
    and the classical bound for $t \geq 0$, $\cos t \geq 1 - \frac{t^2}{2}$,
    \begin{align*}
        I_0(u) 
        &\geq \int_0^{\frac{1}{\pi \sqrt{u}}} e^{u\cos\pi \theta} \, d\theta \\
        &\geq  \frac{1}{\pi\sqrt{u}}e^{u(1-\frac{1}{2u})} \\
        &\geq \frac{e^u}{2\pi\sqrt{u}}.
    \end{align*}

Moreover, using the integral representation of $I_0$ and differentiating the map $u \mapsto e^{-u} I_0(u)$, one obtains for $u>0$:

    \[
    \left( e^{-u}I_0(u)\right)' = \int_{0}^{1} \left(\cos\pi\theta-1\right)e^{u\left(\cos\pi\theta-1\right)}\,d\theta < 0.
    \]
\end{proof}

With these two technical lemmas established, we proceed to estimate $P_1$.

\begin{lemma} \label{lemma:MiddleSumOfEstimate}
    For $s\in \R$, $2\leq s \leq \dfrac{\log p}{\log_3 p}$, the following holds\textnormal{:}

    \[\sum_{ \frac{s}{\pi}  < |j| \leq (\log p)^2} \log\left(\frac{1}{d}\sum_{k=0}^{d-1} \exp\left(\frac{2s\cos\left( \frac{2k\pi}{d} \right)}{\pi(2j+1)}\right)\right) 
    = \frac{2s}{\pi}\int_{0}^1  \dfrac{\alpha_d(u)}{u^2} \,du + O(1),\]
\noindent
    where $\alpha_d$ is defined in~\eqref{eq:def_alpha_d} and $\displaystyle\int_{0}^1  \dfrac{\alpha_d(u)}{u^2} \,du \approx 0.49$ when $d=2$, and $\approx 0.25$ when $d>2$.
    
\end{lemma}

\begin{proof}[Proof of Lemma~\ref{lemma:MiddleSumOfEstimate}]

Let for $j\in\Z$, $x_j := \dfrac{2s}{\pi(2j+1)}$, by using \eqref{eq:property3-alphad} of Lemma~\ref{lemma:estimates_alpha_d}, we have
 \[
 \int_{(\log p)^2}^{\infty} \alpha_d\left( \dfrac{2s}{\pi (2t+1)} \right) \,dt = \int_{0}^{s/\pi ((\log p)^2+1/2)} \dfrac{\alpha_d(u)}{u^2} \,du  = O\!\left( \frac{s}{(\log p)^2}  \right) = o(1).
 \]
Moreover, by the monotonicity of $\alpha_d$ on $(0,\infty)$ and the upper bound in~\eqref{eq:property1-alphad}, we get the following inequalities:

\[
\int_{ \frac{s}{\pi} }^{(\log p)^2+1} \alpha_d\left( \dfrac{2s}{\pi (2t+1)}\right)dt
\leq \sum_{j = \left \lfloor \frac{s}{\pi} \right \rfloor+1}^{(\log p)^2} \alpha_d\left(x_j \right) 
\leq \alpha_d(1)+ \int_{ \frac{s}{\pi} }^{(\log p)^2} \alpha_d\left( \dfrac{2s}{\pi (2t+1)}\right)dt.
\]
\noindent
We remark that $x_{1-j} = x_j$, combining that with the comparison between the sum 
and the integral gives us

\begin{align*}
    \sum_{ \frac{s}{\pi}  < |j| \leq (\log p)^2} \alpha_d\left( x_j \right) 
    &= 2\sum_{ \frac{s}{\pi}  < j \leq (\log p)^2} \alpha_d\left( x_j \right) + O(1) \\
    &= 2\int_{\frac{s}{\pi}}^{(\log p)^2} \alpha_d\left( \dfrac{2s}{\pi (2t+1)}\right) \,dt + O(1) \\
    &= \frac{2s}{\pi}\int_0^{1 - \pi/(2s+\pi)} \dfrac{\alpha_d(u)}{u^2} \,du + O(1) \\
    &= \frac{2s}{\pi}\int_0^{1} \dfrac{\alpha_d(u)}{u^2} \,du -  \frac{2s}{\pi}\int_{2s/(2s+\pi) }^{1} O(1) \,du + O(1) \\
    &= \frac{2s}{\pi}\int_0^{1} \dfrac{\alpha_d(u)}{u^2} \,du + O(1).
\end{align*}
    
\end{proof}

\begin{lemma} \label{lemma:HeadailOfEstimate}
    For $s\in \R, 2\leq s \leq \dfrac{\log p}{\log_3 p}$, the following holds for even $d$\textnormal{:}

    \[\sum_{ |j| \leq s/\pi } \log\left(\frac{1}{d}\sum_{k=0}^{d-1} \exp\left(\frac{2s\cos\left( \frac{2k\pi}{d} \right)}{\pi(2j+1)}\right)\right) = \frac{2s}{\pi}\log s + \hat{C}_d s+ O(\log s),\]
    
    \noindent
    with $\hat{C}_d = \dfrac{2}{\pi}\left(\gamma + \log\dfrac{4}{\pi}\right) + \dfrac{2}{\pi}\displaystyle\int_1^\infty \dfrac{\alpha_d(u) - u}{u^2}\,du $.

\end{lemma}

\begin{proof}[Proof of Lemma~\ref{lemma:HeadailOfEstimate}]

Let $d\geq 2$ an even integer. We recall that $x_j = \dfrac{2s}{\pi(2j+1)}$ for $j\in\Z$. 
Lemma~\ref{lemma:estimates_alpha_d} asserts that the map $u \mapsto \alpha_d(u)-u$ is non-decreasing, and Lemma~\ref{lemma:boundsI0} ensures the logarithmic bound $|\alpha_d(u)-u| = O(\log|u|)$. 
Consequently,

    \begin{align} \label{eq:lemma5.8-part1}
        \begin{split}
            \sum_{ 0 \leq |j| \leq \frac{s}{\pi}} \alpha_d(x_j) 
        &= 2\sum_{ j=0}^{\frac{s}{\pi}} \alpha_d(x_j) + O(1) \\
        &= 2\sum_{ j=0}^{\frac{s}{\pi}} \dfrac{2s}{\pi(2j+1)} + 2\int_{0}^{\frac{s}{\pi} - \frac{1}{2}} \alpha_d\left( \dfrac{2s}{\pi (2t+1)} \right) - \dfrac{2s}{\pi (2t+1)}  \,dt + O(\log s) \\
        &= \frac{2s}{\pi}\left(\sum_{ j=0}^{\frac{s}{\pi}} \dfrac{1}{j+\frac{1}{2}} +\int_{1}^{\infty} \frac{\alpha_d(u) - u}{u^2}\,du - \int_{\frac{2s}{\pi} }^{\infty} \dfrac{\alpha_d(u) - u}{u^2}\,du \right) + O(\log s). 
        \end{split}
    \end{align}

\noindent
    Using Lemma~\ref{lemma:HarmonicEstimation}, we have 
    \begin{equation} \label{eq:lemma5.8-harmonicSeries}
        \sum_{ j=0}^{\frac{s}{\pi}} \dfrac{1}{j + \frac{1}{2}} = \log{s} + \gamma + 2\log{2} - \log{\pi} + O\!\left(\frac{1}{s}\right).
    \end{equation}
\noindent
    For $u\geq 0$, let
\[
E(u):=
\frac{1}{d}\sum_{k=0}^{d-1}
\exp\!\left(u \cos\!\left(\frac{2k\pi}{d}\right)\right)
- I_0(u),
\]
and $\gamma_u:\theta\mapsto \exp\!\left(u \cos\!\left(2\pi\theta\right)\right)$. $\gamma_u$ is decreasing on $[0,\tfrac12]$ and increasing on $[\tfrac12,1]$, so we have

\begin{align*}
    E(u) &= \sum_{k=0}^{\frac{d}{2}-1}\int_{k/d}^{(k+1)/d} \left(\gamma_u(\tfrac{k+1}{d}) - \gamma_u(\theta) \right)\,d\theta 
    + \sum_{k=\frac{d}{2}+1}^{d-1}\int_{k/d}^{(k+1)/d} \left(\gamma_u(\tfrac{k}{d}) - \gamma_u(\theta) \right)\,d\theta \\
    &\;\;+ \frac{\gamma_u(0)}{d} - \int_{1/2}^{1/2 + 1/d}  \gamma_u(\theta)\,d\theta \\
    &\leq \frac{e^u}{d},
\end{align*}
since the first, second and fourth terms are nonpositive, hence they can be bounded above by $0$. 
It is easy to see that $E(u)\geq 0$ with~\eqref{eq:serieAvecCoeffBessel-unityRoot}.
Next, we use the identity

\begin{equation} \label{eq:lemma5.8-identityAlpha_d}
\alpha_d(u)-u = \log I_0(u)-u + \log\!\left(1 + \frac{E(u)}{I_0(u)}\right)
\end{equation}

to derive (by using Lemma~\ref{lemma:boundsI0})

\begin{align}  \label{eq:lemma5.8-errorTerm}
\begin{split}
    \left| \int_{\frac{2s}{\pi}}^{\infty} \dfrac{\alpha_d(u)-u}{u^2} \,du \right|
&\leq 
\int_{\frac{2s}{\pi}}^{\infty}
O\!\left(\frac{\log u}{u^2}\right)\,du + \int_{\frac{2s}{\pi}}^{\infty}
O\!\left(\frac{ \log\!\left(1+\frac{2\pi\sqrt{u}}{d}\right)}{u^2}\right)\,du   \\ 
&=O\!\left(\frac{\log s}{s}\right).
\end{split}
\end{align}

    Combining equations~\eqref{eq:lemma5.8-part1}, \eqref{eq:lemma5.8-harmonicSeries} and~\eqref{eq:lemma5.8-errorTerm} yields

    \begin{align*}
        \sum_{ 0 \leq |j| \leq \frac{s}{\pi}} \alpha_d\left(  \dfrac{2s}{\pi(2j+1)} \right) 
       &= \frac{2s}{\pi}\log s  +\hat{C}_d s  + O(\log s),
    \end{align*}
and $\hat{C}_d =   \dfrac{2}{\pi} \left(\gamma + \log{\dfrac{4}{\pi}}+ \displaystyle\int_{1}^{\infty} \dfrac{\alpha_d(u) - u}{u^2}\,du \right)$.

\end{proof}

\begin{proof}[Proof of Proposition~\ref{prop:estimateLogEsperance}]
    The proposition follows from  Lemmas~\ref{lemma:TailOfEstimate}, \ref{lemma:MiddleSumOfEstimate} and \ref{lemma:HeadailOfEstimate}.
\end{proof}

\begin{lemma} \label{lemma:boundsOnC_d}
 The sequence $(C_{2d})_{d\ge 1}$ (defined in Proposition~\ref{prop:estimateLogEsperance}) is decreasing. Furthermore,
\[
C_{2} = 0.1029\cdots,
\]
and
\[
\lim_{d\to\infty} C_{2d} = \frac{2}{\pi}\left( \gamma + \log\frac{4}{\pi} + \int_0^1 \frac{\log I_0(u)}{u^2} \,du + \int_1^\infty \frac{\log I_0(u) - u}{u^2} \,du   \right) = -0.1722\cdots.
\]
\end{lemma}

\begin{proof}[Proof of Lemma~\ref{lemma:boundsOnC_d}]
   
For $u\geq 0$,  
\[
    \alpha_d(u) = \frac{1}{d}\sum_{k=0}^{d-1} \exp\left(   u\cos\left( \frac{2k\pi}{d}\right) \right)= I_0(u) + 2 \sum_{n\geq 1} I_{nd}(u).
\]

\noindent
For even $d$, the function $I_{nd}$ is even and nonnegative for all $n\in\mathbb{N}$.  
Moreover, the map $d \mapsto I_{nd}(u)$ is decreasing and tends to $0$ as $d\to\infty$.  
This implies that $d \mapsto \alpha_d$ is decreasing, and consequently that the sequence $d \mapsto C_{2d}$ is also decreasing.

\medskip

Now, let us assume $d$ is large. Then, for $n\geq 1$ and $0\leq u\leq \sqrt{d}$, we can use the 
asymptotic estimate of $I_{nd}(u)$ for large orders (see \cite{Olver1974}, p. 374):

\begin{equation} \label{eq:EstimateIn}
  I_{nd}(u) = \dfrac{u^{nd}}{2^{nd} (nd)!} \left( 1 + O\!\left(\frac{u^2}{nd}\right) \right). 
\end{equation}

\noindent
From equation~\eqref{eq:EstimateIn}, for $0 \leq u \leq 1$ we have 
\[
\alpha_d(u) = \log \left(I_0(u) +  O\!\left(\frac{u^d}{d!}\right)\right).
\]
Thus, since $1\leq I_0(u) \leq e$ for $u$ in this range,
\begin{align}\label{eq:headIntI0-estimate}
\begin{split}
    \int_{0}^1  \dfrac{\alpha_d(u)}{u^2} \,du &= \int_{0}^1  \dfrac{\log I_0(u)}{u^2} \,du + \int_{0}^1  O\!\left(\frac{u^{d-2}}{d!} \right) \,du\\
    &=   \int_{0}^1  \dfrac{\log I_0(u)}{u^2} \,du +  O\!\left(\frac{1}{d!} \right).
\end{split}
\end{align}

For $u\geq 1$, from equations~\eqref{eq:lemma5.8-identityAlpha_d},~\eqref{eq:lemma5.8-errorTerm} and the bound $\ln (1+x) \leq x$ for $x\geq0$, we get
\begin{align} \label{eq:middleIntI0-estimate}
\begin{split}
    \int_{1}^{\infty} \dfrac{\alpha_d(u)-u}{u^2} \,du 
    &= \int_{1}^{\infty} \dfrac{\log I_0(u)-u}{u^2}\,du +
    \int_{1}^{\infty} O\!\left(\frac{ \frac{2\pi\sqrt{u}}{d}}{u^2}\right)\,du \\
    &= \int_{1}^{\infty} \dfrac{\log I_0(u)-u}{u^2}\,du +
     O\!\left(\frac{ 1}{d} \right) \\
\end{split}
\end{align}

Combining~\eqref{eq:headIntI0-estimate} and \eqref{eq:middleIntI0-estimate} gives us the desired result.

\end{proof}

\begin{remark} 
    The estimate for small odd moduli is insufficient in this case: unlike the even-order setting, these computations do not yield a lower bound matching the upper bound. 
To tackle this problem, we adopt a different approach in Section~\ref{section6:proofoddLowerBound}. Although the desired lower bound could in principle be obtained by taking the absolute value instead of the real part of the sum, this leads to substantially more delicate estimates, beyond the scope of the present methods. 
We also note that it remains an interesting open problem to determine whether $(C_{2d+1})$ is increasing in $d$; numerical experiments are consistent with this conjecture (see the plots in the appendix).

\end{remark}

\medskip

Now that we have computed the expected value of the Laplace transform of the random model, what is left to do is using the saddle-point method to get the desired 
lower bound for $\Phi_\chi(V)$. We have the following estimate for $s\leq \frac{\log p}{100 (\log_2 p)^2}$:

\begin{equation}  \label{eq:SaddlepointIntegral}
    \int_{-\infty}^{\infty} e^{st}\tilde{\Phi}_\chi(t)\,dt   = \exp{\!\left( \frac{2}{\pi}s\log{s} + C_d s + O(\log s) \right)},
\end{equation} 
where $C_d$ is defined in Proposition~\ref{prop:estimateLogEsperance}.

\medskip

Now, if we take $V \leq \frac{2}{\pi} (\log_2 p - 2 \log_3 p)$ a large real number, we shall chose $s$ by using the saddle-point method, in other words the $s$ such that
\begin{equation*}
\label{saddlepoint}
\left(\frac{2}{\pi}s\log{s} + C_d s - Vs \right)' = 0 \Leftrightarrow s = \exp{\!\left(\frac{\pi}{2}(V-C_d) -1\right)}.
\end{equation*}

\noindent
Let $\varepsilon= \frac{\log s}{\sqrt{s}}$ and $S = se^\varepsilon$, it follows from~\eqref{eq:SaddlepointIntegral} that

\begin{align*}
\int_{V+ \frac{2\varepsilon}{\pi}}^{\infty} e^{st}\tilde{\Phi}_\chi(t)\,dt 
&\leq \exp{\!\left(s(1-e^\varepsilon)\left(V+\frac{2\varepsilon}{\pi}\right)\right)}  
 \int_{V+ \frac{2\varepsilon}{\pi}}^{\infty} e^{St}\tilde{\Phi}_\chi(t)\,dt   \\
&\leq \exp{\!\left(s(1-e^\varepsilon)\left(V+\frac{2\varepsilon}{\pi}\right)+ \frac{2}{\pi}s e^\varepsilon \log (s e^\varepsilon) + C_d se^\varepsilon + O(\log s) \right)  }  \\
&\leq \exp{\!\left( \frac{2}{\pi}(1+\varepsilon - e^\varepsilon)s + O(\log s) \right)  }  \int_{-\infty}^{\infty} e^{st}\tilde{\Phi}_\chi(t)\,dt.
\end{align*}

\noindent
Similarly, we get (if we now put $S = se^{-\varepsilon}$)
\begin{align*}
\int_{-\infty}^{V- \frac{2\varepsilon}{\pi}} e^{st}\tilde{\Phi}_\chi(t)\,dt 
&\leq \exp{\!\left(s(1-e^{-\varepsilon})\left(V-\frac{2\varepsilon}{\pi}\right)\right)}  
 \int_{-\infty}^{V- \frac{2\varepsilon}{\pi}}e^{St}\tilde{\Phi}_\chi(t)\,dt   \\
&\leq \exp{\!\left(s(1-e^{-\varepsilon})\left(V - \frac{2\varepsilon}{\pi}\right)+ \frac{2}{\pi}s e^{-\varepsilon} \log (s e^{-\varepsilon}) +
 C_d se^{-\varepsilon} + O(\log s) \right)  }  \\
&\leq \exp{\!\left( \frac{2}{\pi}(1-\varepsilon - e^{-\varepsilon})s + O(\log s) \right)  }  \int_{-\infty}^{\infty} e^{st}\tilde{\Phi}_\chi(t)\,dt.
\end{align*}

\medskip
Combining the previous two inequalities gives us the following estimate:

\begin{align*}
    \int_{V- \frac{2\varepsilon}{\pi}}^{V+ \frac{2\varepsilon}{\pi}} e^{st}\tilde{\Phi}_\chi(t)\,dt
    &= \left(1 - 2e^{- \frac{(\log s)^2}{2} + O(\log s)} \right) \exp{\!\left( \frac{2}{\pi}s\log{s} + C_d s + O(\log s)   \right)} \\
    &= \exp{\!\left( \frac{2}{\pi}s\log{s} + C_d s + O(\log s)   \right)}.
\end{align*}

\noindent
Since $\tilde{\Phi}_\chi(t)$ is a non increasing function of $t$, we have the following estimate:

\begin{align*}
    \tilde{\Phi}_\chi\left(V-\frac{2\varepsilon}{\pi}\right) &\le \exp \left(-\frac{2}{\pi}\exp\left(\frac{\pi}{2}(V-C_d) -1 \right)\left(1 + O\!\left( e^{-\frac{\pi}{4}V} \right) \right) \right) \\
    \tilde{\Phi}_\chi\left(V+\frac{2\varepsilon}{\pi}\right) &\ge \exp \left(-\frac{2}{\pi}\exp\left(\frac{\pi}{2}(V-C_d) -1 \right)\left(1 + O\!\left( e^{-\frac{\pi}{4}V} \right) \right) \right).
\end{align*}

\medskip
We have the following bound:
\[\max_{x\in[0,1]} \left|\frac{f_\chi(e_p(K+x))}{\sqrt{p}}\right|  \geq \left|\frac{f_\chi(e_p(K+\frac{1}{2})}{\sqrt{p}}\right|,
\]
\noindent
and since in $\tilde{\Phi}_\chi(V)$ we are counting values of $K\in \F_p \setminus \mathcal{E}_p$ such that the modulus of our function exceed $V$, with the definition of $\Phi_\chi(V)$,
we see that $\Phi_\chi(V) \geq \tilde{\Phi}_\chi(V) $, so that gives us

\begin{equation} \label{eq:EstimationEnUnDemi}
\exp\left(-C_d^-\exp{\!\left(\frac{\pi}{2}V \right)}\left(1+O(e^{-\frac{\pi}{4} V})\right)\right) \le \Phi_\chi(V),
\end{equation}
\noindent
with $C_d^- := \frac{2}{\pi}\exp( -\frac{\pi}{2}C_d -1) $.

\medskip
Taking $d=2$ (hence studying Fekete polynomials) and noticing that $\lvert \mathcal{E}_p\rvert~\ll~p^{3/4}$ 
gives us Theorem~\ref{theorem:distribution1/2CharacterSum}.

\medskip

\section{Proof of the lower bound for the odd character case}
\label{section6:proofoddLowerBound}
\medskip

If we attempt to derive a lower bound for odd-order characters using the same method 
as in the even-order case, a discrepancy arises: one obtains instead a lower bound of the form

\[
\exp\left(-C\exp\left(\frac{\pi}{1+\cos\frac{\pi}{d}}V  \right)    \right).
\]
Thus, the lower bound would not be optimal for $d\ll V^{1/2+\varepsilon}, \, \varepsilon>0$, because as we will see in this section, the correct order of magnitude is the one of the upper bound.
In line with the strategy employed in Section~\ref{section3:proofUpperBound}, we will show that for any fixed integer $n$, there exists many $K\in\F_p$ such that the 
$n$-tuple $\left(\chi(K), \chi(K+1),\ldots,\chi(K+n) \right)$ realizes any given fixed pattern. 
The following proposition will quantify and generalize this idea.

\begin{prop} Let $\chi \!\pmod p$ be a character of order $d$. Uniformly for $n\leq p$,   $m_1,\ldots,m_n\in~\F_p$ distinct values and $(a_j)\in(\U_d)^{n}$,
    we have
\label{prop:WeilFixValues}
\[ \frac{1}{p} \Big| \big\{ K \in \F_p :\ \forall j \le n, \; \chi(K+m_j)=a_j \big\} \Big| = \dfrac{1}{d^{n}} + O\!\left(\frac{n}{\sqrt{p}}\right).
\]
\end{prop}

\begin{proof}
    For a fixed $m\in\F_p$ and $a\in\U_d$, we have
    \[
    \frac{1}{d}\sum_{t=0}^{d-1} \left( \chi(m) \overline{a}\right)^t = \mathds{1}_{\chi(m) = a}.
    \] 
\noindent
    From this identity, we deduce that the cardinality of the set in Proposition~\ref{prop:WeilFixValues} can be written as
    
    \[ 
     \sum_{K=1}^p \dfrac{1}{d^n}\prod_{1\le j \le n} \left( \sum_{t=0}^{d-1} \left( \chi(K+m_j) \overline{a_j}\right)^t \right).  
     \]
\noindent
    We then develop the inner product:

    \[
    \prod_{j=1}^{n} \left( \sum_{t=0}^{d-1} \left( \chi(K+m_j) \overline{a_j}\right)^t \right) 
    = \sum_{0 \leq t_1, \cdots, t_n < d} \prod_{j=1}^{n} \overline{a_j}^{t_j} \chi\left(\textstyle \prod_{j=1}^{n}(K+m_j)^{t_j}   \right).
    \]
\noindent
    From this expression and by using Lemma~\ref{WeilBound}, we deduce

    \begin{align*}
        &\frac{1}{p} \Big| \big\{ K \in \F_p :\ \forall j \le n, \chi(K+m_j)=a_j \big\} \Big| \\
        = \;& \frac{1}{p}\sum_{K=1}^p \dfrac{1}{d^n}\prod_{1\le j \le n} \left( \sum_{t=0}^{d-1} \left( \chi(K+m_j) \overline{a_j}\right)^t \right) \\
        = \;& \dfrac{1}{pd^n}\sum_{0 \leq t_1, \cdots, t_n < d} \prod_{j=1}^{n} \overline{a_j}^{t_j} \sum_{K=1}^p\chi\left(\textstyle \prod_{j=1}^{n}(K+m_j)^{t_j}   \right) \\
        = \;& \dfrac{1}{d^n} + \dfrac{1}{pd^n} \sum_{\substack{0 \leq t_1, \cdots, t_n < d \\ (t_1,\ldots,t_n)\neq (0,\ldots,0)}} \prod_{j=1}^{n} \overline{a_j}^{t_j} \sum_{K=1}^p\chi\left(\textstyle \prod_{j=1}^{n}(K+m_j)^{t_j}   \right) \\
        = \;& \frac{1}{d^n} + O\!\left(\frac{n}{\sqrt{p}} \right).
    \end{align*}
\end{proof}

Now, we just have to show that among the set of $K$ such that $\chi(K+j)$ follows the 
desired pattern, most of them are such that $\sum_{J\leq |j| <p/2} \chi(K-j)c_j$ is 
not too large.

\begin{theorem}
    \label{theorem:OddCharacterLowerBound}
    Let $p$ be a large prime, and $\chi \pmod p$ a Dirichlet character of odd order $d$. 
    For all real number $ 1\leq V \leq \frac{\delta_d}{\pi}(\log_2 p - 2\log_3 p) - 7\sqrt{\log d}$ we have

    \begin{equation*}
        \exp\left( - \tilde{C}_d^- \exp \left(\frac{\pi}{\delta_d}V  \right)  \left(  1 + O(e^{-\pi V/(2\delta_d)}) \right)   \right) \leq \Phi_\chi(V),
    \end{equation*}
    where $\delta_d = 2\cos\frac{\pi}{2d}$ and 
    \begin{equation} \label{eq:valCd-smallOdd}
    \tilde{C}_d^- = \frac{\log d}{2}\exp\left(\frac{7\pi}{\delta_d}\sqrt{\log d} - \gamma\right),
    \end{equation}
     with $\gamma$ the Euler–Mascheroni constant.
    \end{theorem}

\begin{proof}
    Take $V\geq 1$ and let $\lambda = 7\sqrt{\log d}$. 
    Let $J\in\R$ and $k\in\N$ be parameters to be chosen later, then: 
    \begin{align*} 
    &\frac{1}{p} \left| \left\{K\in\F_p :\; \Big|\sum_{J\leq |j|< p/2} \chi(K-j)c_j\Big| \geq \lambda \right \} \right| \\
    \leq & \frac{\lambda^{-2k}}{p} \sum_{K\in\F_p} \left| \sum_{J \leq |j| < \frac{p}{2}}  \chi(K-j)c_j \right|^{2k}   \\
    \ll &  \left(\dfrac{8k}{J\lambda^2} \right)^{k}  + \dfrac{2k}{\sqrt{p}} \left(\frac{\log p}{2\lambda} \right)^{2k},
    \end{align*}
    where $c_j$ is defined in the proof of Lemma~\ref{lemma:productofArithmeticToProbabilistic}, and the last inequality follows from Corollary~\ref{coro:MajorationSommeZMN-CasParticulier}.
    We want this quantity to be negligible in front of $\dfrac{1}{d^{2J+1}}$, where $J$ is chosen in a way that allows
    \[
    \left|\sum_{|j| \leq J} 
    \chi(K-j)c_j \right| \geq V+\lambda.
    \]
    One suitable $J$ would be $J= \exp\left(\frac{\pi}{\delta_d}(V+\lambda) - \gamma - 2\log 2\right)$, as seen in Lemma~\ref{lemma:maxOfSumUnitCircle}.
    Now, if we set $k = \frac{49\log d}{8e}J$, then:
    \[
     \left(\dfrac{8k}{J\lambda^2} \right)^{k}  + \dfrac{2k}{\sqrt{p}} \left(\frac{\log p}{2\lambda} \right)^{2k} = o\!\left(  \frac{1}{d^{2J+1}}  \right).
\]

Thus, we can derive a lower bound by noticing that for having $2\lvert g_{\chi,K}\rvert \geq V$, it is enough to have
\[
\left|\sum_{|j| \leq J}  \chi(K-j)c_j \right| \geq V+\lambda \qquad \text{and} \qquad \left|\sum_{J\leq |j|< p/2} \chi(K-j)c_j(x)\right| \leq \lambda.
\]
\noindent
This translates into the following bound, where the $a_j$ are the ones chosen in the proof of Lemma~\ref{lemma:maxOfSumUnitCircle},
 and by using Corollary~\ref{coro:MajorationSommeZMN-CasParticulier} and Proposition~\ref{prop:WeilFixValues}:

\begin{align}
         \begin{split}
            \Phi_\chi(V) 
            \geq \;&         \frac{1}{p} \Big| \big\{ K \in \F_p :\ \forall |j| \le J, \chi(K- j)=a_j \big\} \Big| \\
- \;& \frac{1}{p} \left| \left\{K\in\F_p :\; \left|\sum_{J\leq |j|< p/2} \chi(K-j)c_j(x)\right| \geq \lambda \right \} \right| \\
            \gg \;& \frac{1}{d^{2J+1}}   \\
            \gg \;& \exp\left( -\tilde{C}_d^- \exp\left( \frac{\pi}{\delta_d}V \right)   \right).  \\
         \end{split}
\end{align}

\end{proof}

\section*{Appendix A. Simulations}

This appendix presents numerical experiments illustrating the main results of the paper.  
The simulations are intended to provide empirical support for the theoretical distributional results obtained in Theorems~\ref{theorem:distributionEvenOrderCharacterSum} and~\ref{theorem:distributionOddOrderCharacterSum}.
More specifically, we only compute a lower bound by evaluating the polynomial at the midpoint of the subarcs, which is conjectured to be the true order of magnitude.
Guided by this heuristic, we conjecture that the tail 
distribution $\Phi_\chi$ exhibits an ordering with respect to the character order: 
for fixed \(V\) in the range of the theorem, the function \(\Phi_\chi(V)\) decreases as the (even) 
order of \(\chi\) increases, increases as the (odd) order increases, and moreover satisfies
\[
\Phi_{\chi_{\mathrm{odd}}}(V) \;\le\;\Phi_{\chi_{\mathrm{even}}}(V)
\]
for any pair of characters \(\chi_{\mathrm{even}},\chi_{\mathrm{odd}}\) of respectively even and odd order modulo $p$.

In the following plots, $\chi_n$ denotes a character of order $n$, and all the characters 
are taken mod~$p = 20 000 821$.
 
\begin{figure}[!htb]
  \centering
  \begin{subfigure}{0.48\textwidth}
    \centering
    \includegraphics[width=\textwidth]{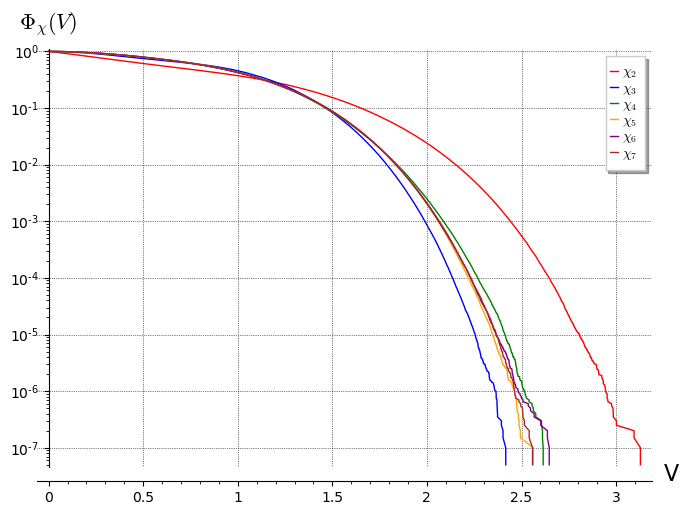}
    \caption{log-log plot of $\Phi_\chi$ for $\chi$ of order ranging from $2$ to $7$.}
  \end{subfigure}
  \hfill
  \begin{subfigure}{0.48\textwidth}
    \centering
    \includegraphics[width=\textwidth]{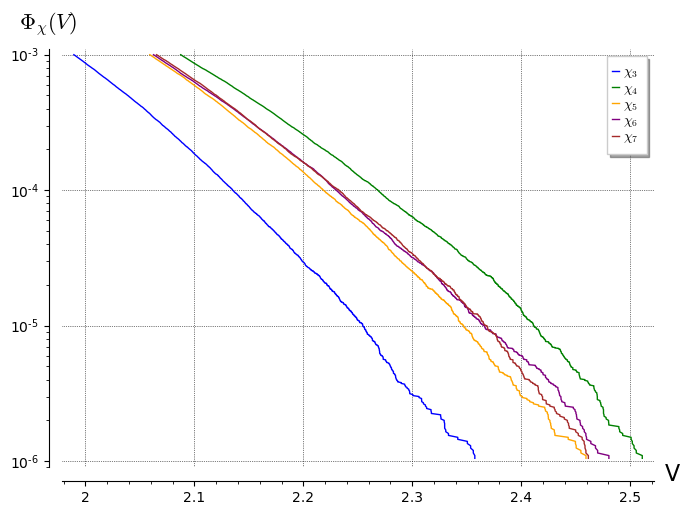}
    \caption{Zoomed version of the first plot, without Legendre's symbol.}
  \end{subfigure}
  \caption{Plot of $\Phi_\chi$ for $\chi \!\mod p=20 000 821$ of order $2,3,4,5,6$ and $7$.}
  \label{fig:twoplots}
\end{figure}

\medskip

These simulations were carried out using SageMath on the computing servers of the \textit{Institut Élie Cartan de Lorraine} (80 cores, 192 GB of RAM).

\providecommand{\bysame}{\leavevmode\hbox to3em{\hrulefill}\thinspace}
\providecommand{\MR}{\relax\ifhmode\unskip\space\fi MR }

\providecommand{\MRhref}[2]{%
  \href{http://www.ams.org/mathscinet-getitem?mr=#1}{#2}
}
\providecommand{\href}[2]{#2}

\vspace{0.5cm}
\noindent
\textsc{Universit\'e de Lorraine, CNRS, Institut Elie Cartan de Lorraine, UMR 7502, Vandoeuvre-l\`es-Nancy, F-54506, France} \\
\noindent
\textit{Email address}: \textbf{\texttt{amine.iggidr@univ-lorraine.fr}}

\end{document}